\documentclass[11pt,reqno]{amsart}

\usepackage[T1]{fontenc}
\usepackage[utf8]{inputenc}
\usepackage{lmodern}
\usepackage[british]{babel}
\usepackage{mathtools,amssymb,amsthm}
\usepackage{array}
\usepackage{microtype}
\usepackage[hidelinks]{hyperref}
\hypersetup{
  pdftitle={Sharp Wasserstein stability and small-delay thresholds for evolution equations with measure-valued principal-part delays},
  pdfauthor={Hiroki Ishizaka},
  pdfkeywords={measure-valued delay, principal-part delay, Wasserstein stability, solution-level sharpness, small-delay instability, kernel quadrature}
}
\usepackage[nameinlink,capitalize,noabbrev]{cleveref}

\numberwithin{equation}{section}

\theoremstyle{plain}
\newtheorem{theorem}{Theorem}[section]
\newtheorem{lemma}[theorem]{Lemma}
\newtheorem{proposition}[theorem]{Proposition}
\newtheorem{corollary}[theorem]{Corollary}

\theoremstyle{definition}

\theoremstyle{remark}
\newtheorem{remark}[theorem]{Remark}

\crefname{theorem}{theorem}{theorems}
\Crefname{theorem}{Theorem}{Theorems}
\crefname{lemma}{lemma}{lemmas}
\Crefname{lemma}{Lemma}{Lemmas}
\crefname{proposition}{proposition}{propositions}
\Crefname{proposition}{Proposition}{Propositions}
\crefname{corollary}{corollary}{corollaries}
\Crefname{corollary}{Corollary}{Corollaries}
\crefname{definition}{definition}{definitions}
\Crefname{definition}{Definition}{Definitions}
\crefname{assumption}{assumption}{assumptions}
\Crefname{assumption}{Assumption}{Assumptions}
\crefname{remark}{remark}{remarks}
\Crefname{remark}{Remark}{Remarks}
\crefname{example}{example}{examples}
\Crefname{example}{Example}{Examples}

\newcommand{\R}{\mathbb{R}}
\newcommand{\M}{\mathcal{M}}
\newcommand{\W}{\mathcal{W}}
\newcommand{\K}{\mathcal{K}}

\newcommand{\dd}{\,\mathrm{d}}
\newcommand{\Tend}{\mathfrak{T}}
\newcommand{\norm}[1]{\left\lVert #1 \right\rVert}
\newcommand{\abs}[1]{\left\lvert #1 \right\rvert}
\newcommand{\pair}[2]{\left\langle #1,#2 \right\rangle}

\title[Sharp Wasserstein stability for principal-part delays]
{Sharp Wasserstein stability and small-delay thresholds for evolution equations with measure-valued principal-part delays}
\thanks{Affiliation: Team FEM, Matsuyama, Japan.
Email: \texttt{h.ishizaka005@gmail.com}.}
\thanks{The weighted well-posedness and qualitative kernel-stability results recalled in \Cref{sec:framework} are established in the companion article~\cite{IshizakaMixed2026}, of which the present quantitative and sharpness theory is a development.}

\author{Hiroki Ishizaka}


\begin{document}
\raggedbottom

\begin{abstract}
We study the dependence of solutions to coercive evolution equations on the law of a measure-valued delay acting through the principal spatial operator.  For non-negative delay measures of equal mass, transport couplings and the time-translation modulus of one reference trajectory yield quantitative kernel-stability estimates.  Fractional temporal regularity gives a Wasserstein rate, with explicit glueing and data criteria for its application.  For moving point delays in a diffusion model, a solution-level lower bound matches both the rate and its dependence on the kernel mass.  This matching estimate holds for zero prescribed history and uniformly bounded forcing in the natural dual energy space.  The results also quantify concentration to positive delays and approximation by atomic quadrature.  Separately, we determine a sharp class-level threshold for energy bounds uniform in the delay location.  At the critical threshold, any fixed scalar mode is spectrally stable, but uniform space--time energy control fails in a small-delay, high-frequency limit.  Above the threshold, the growth has an explicit asymptotic logarithmic rate.
\end{abstract}

\maketitle

\section{Introduction}\label{sec:introduction}
Let $\Tend>0$ and $\tau>0$.  On a Gelfand triple $V\hookrightarrow H\cong H^*\hookrightarrow V^*$, we consider
\begin{align}\label{eq:intro-problem}
  \partial_tu+A_0u+\int_{[0,\tau]}A_1\tilde u(t-s)\,\dd\mu(s)=f,
  \quad u(0)=u_0,
\end{align}
where $A_0,A_1:V\to V^*$ are induced by bounded forms, the instantaneous form is coercive, and $\tilde u$ joins the solution to its prescribed history.  The measure $\mu$ covers distributed memory and discrete delays~\cite{HaleVerduynLunel1993,BatkaiPiazzera2005}.  We study the change in the solution when the delay measure is perturbed.

The companion paper~\cite{IshizakaMixed2026} establishes weighted well-posedness and qualitative kernel stability, but not rates in a transport metric.  For non-negative delay measures $\mu$ and $\nu$ of common mass $M$, we combine their residual estimate with a transport coupling.  The resulting bound depends on the time-translation modulus of one reference trajectory (\Cref{thm:transport-modulus}).  When the joined trajectory extends to $G\in H^\theta(\R;V)$ for some $0<\theta\le1$, \Cref{thm:wasserstein-rate} gives
\begin{align*}
  \norm{u_\mu-u_\nu}_{L^2(0,\Tend;V)}
  \le C\,M^{1-\theta}W_1(\mu,\nu)^\theta\norm{G}_{H^\theta(\R;V)},
\end{align*}
where $W_1$ denotes the equal-mass $1$-Wasserstein transport cost.  We give conditions on the history and data that ensure the required regularity (\Cref{prop:Htheta-gluing,prop:data-H1-criterion}).  This combines transport estimates~\cite{AmbrosioGigliSavare2008,Villani2009} with fractional time-translation bounds~\cite{DiNezzaPalatucciValdinoci2012}.  For moving point delays in a diffusion model, the matching lower bound holds for actual solution differences, with zero prescribed history and uniformly bounded forcing (\Cref{thm:solution-sharpness} and \Cref{cor:two-sided-modulus}).  It proves optimality of both the exponent and the mass dependence within the stated class.

Delay dependence for parabolic equations has been considered under weak-$*$ assumptions~\cite{KryspinMierczynski2023,KryspinMierczynski2024}, for measure-valued delay controls~\cite{CasasMateosTroltzsch2018}, and for reaction--diffusion and nonlinear diffusion equations~\cite{Shikhman2026,Tsukamoto2026}.  Of particular relevance is the sharp state-delay result of Shikhman~\cite{Shikhman2026b}, which gives a logarithmically corrected modulus for fixed-time errors in $H$, as well as bounds for quadrature and regular reference paths~\cite[Remark~6.4]{Shikhman2026b}. Here, the delay acts through a principal-order operator, and errors are measured in $L^2(0,\Tend;V)$.  In this energy space, there is no uniform vanishing modulus over the full $L^2$ reference class; the rate is instead determined by positive temporal regularity.  The distinction is not merely the choice of metric: on a bounded delay interval, $W_1$ and the bounded-Lipschitz distance are equivalent for equal-mass measures.

We also examine estimates uniform in the delay location as delays tend to zero.  With $\alpha_0$ denoting the instantaneous coercivity constant and $\Lambda_1$ the bound for the delayed form, the ratio $\Lambda_1M/\alpha_0$ has the critical value one.  Below it, the energy bound is uniform; at equality, each fixed scalar mode is spectrally stable, but the uniform space--time estimate fails along a small-delay, high-frequency sequence.  Above it, for a positive discrete lag $\tau_d$, the growth rate is asymptotic to $\log(\Lambda_1M/\alpha_0)/\tau_d$ (\Cref{prop:subcritical-mass-bound} and \Cref{thm:small-delay-instability}).  Small-delay destabilisation is known in feedback theory~\cite{Datko1988,LogemannRebarberWeiss1996} and highest-order delayed parabolic equations~\cite{YongPan1993}.  Our focus is the uniform finite-time energy bound, including its critical case.

The transport estimate also bounds the effect of replacing a delay measure by atomic quadrature, separately from the discretisation error for the state.  Earlier numerical work includes averaging schemes~\cite{BanksBurns1978,BanksBurnsCliff1981}, delay discretisations~\cite{BellenZennaro2003,Pao1999}, and methods for Volterra equations~\cite{Lubich1988,LubichOstermann1993,McLeanThomee1993,Thomee2006}.  Mesh-uniform kernel estimates would additionally require uniform time regularity of the semidiscrete reference solutions; this is not implied by discrete coercivity and is left to future work.

\Cref{sec:framework} recalls the framework, \Cref{sec:rates} proves the transport estimates, \Cref{sec:small-delay} analyses the small-delay threshold, and \Cref{sec:sharpness} establishes solution-level optimality.

\section{Framework and residual stability}\label{sec:framework}
The functional framework and the well-posedness, residual, and narrow kernel-stability results in this section are established in the companion article~\cite{IshizakaMixed2026}.  We recall their statements, without proof, because they provide the tools used in \Cref{sec:rates,sec:small-delay,sec:sharpness}.

Let $V$ and $H$ be real separable Hilbert spaces with $V\hookrightarrow H$ continuously and densely.  We identify $H$ with its dual and use the Gelfand triple
\begin{align*}
  V\hookrightarrow H\cong H^*\hookrightarrow V^*.
\end{align*}
Let $a_0,a_1:V\times V\to\R$ be bilinear forms satisfying
\begin{align}
|a_i(w,v)|&\leq \Lambda_i\norm{w}_V\norm{v}_V, \quad i\in\{0,1\},\label{eq:form-bounded}\\
a_0(v,v)&\geq \alpha_0\norm{v}_V^2,\label{eq:a0-coercive}
\end{align}
for constants $\Lambda_0,\Lambda_1,\alpha_0>0$.  We write $A_i\in\mathcal L(V,V^*)$ for the induced operators.

Let $\M([0,\tau])$ be the finite signed Borel measures with total-variation norm, and set
\begin{align*}
\M^0([0,\tau])&:=\{\sigma\in\M([0,\tau]):|\sigma|(\{0\})=0\},\\
\M_+^0([0,\tau])&:=\M^0([0,\tau])\cap\M_+([0,\tau]).
\end{align*}
For $\beta\ge0$, we define
\begin{align}
  \norm{g}_{L^2_\beta(I;X)}
  &:=\left(\int_Ie^{-2\beta t}\norm{g(t)}_X^2\,\dd t\right)^{1/2},\label{eq:weighted-norm}\\
  \kappa_\sigma(\beta)
  &:=\int_{[0,\tau]}e^{-\beta s}\,\dd|\sigma|(s).\label{eq:kappa}
\end{align}

For $u\in L^2(0,\Tend;V)$ and a prescribed history $\psi\in L^2(-\tau,0;V)$, let
\begin{align}\label{eq:joined-trajectory}
  \tilde u(t):=
  \begin{cases}
    \psi(t),&-\tau<t<0,\\
    u(t),&0<t<\Tend.
  \end{cases}
\end{align}
We use
\begin{align*}
  \W(0,\Tend)
  :=\{u\in L^2(0,\Tend;V):\partial_tu\in L^2(0,\Tend;V^*)\},
\end{align*}
equipped with
\begin{align*}
\norm{u}_{\W(0,\Tend)}
  &:=
  \left(
    \norm{u}_{L^2(0,\Tend;V)}^2
    +
    \norm{\partial_tu}_{L^2(0,\Tend;V^*)}^2
  \right)^{1/2}.
\end{align*}
Then, $\W(0,\Tend)\hookrightarrow\mathcal C([0,\Tend];H)$.

\begin{lemma}[Measure-valued history operator]\label[lemma]{lem:measure-convolution}
For $\sigma\in\M([0,\tau])$ and $g\in L^2(-\tau,\Tend;V)$, extended by zero to $\R$, the formula
\begin{align*}
  (\K_\sigma g)(t):=\int_{[0,\tau]}A_1g(t-s)\,\dd\sigma(s)
\end{align*}
defines an element of $L^2(0,\Tend;V^*)$, independent of the chosen strongly measurable representative, and
\begin{align}\label{eq:measure-convolution-bound}
  \norm{\K_\sigma g}_{L^2(0,\Tend;V^*)}
  \leq \Lambda_1\norm{\sigma}_{\M([0,\tau])}
       \norm{g}_{L^2(-\tau,\Tend;V)}.
\end{align}
The same notation is used for $G\in L^2(\R;V)$.
\end{lemma}

\begin{proof}
For the proof, see~\cite[Lemma~2.3]{IshizakaMixed2026}.
\end{proof}

We consider
\begin{align}\label{eq:problem}
  \partial_tu+A_0u+\K_\mu\tilde u=f
  \quad\text{in }V^*\text{ for a.e. }t\in(0,\Tend),
  \quad u(0)=u_0\text{ in }H,
\end{align}
with $u_0\in H$, $f\in L^2(0,\Tend;V^*)$, $\psi\in L^2(-\tau,0;V)$, and $\mu\in\M^0([0,\tau])$.

\begin{remark}[An atom at the origin]\label[remark]{rem:atom-zero}
If $\hat\mu=m_0\delta_0+\mu$ with $|\mu|(\{0\})=0$, then $\K_{\hat\mu}\tilde u=m_0A_1u+\K_\mu\tilde u$.  Thus, the atom at zero must be absorbed into the instantaneous form $a_0+m_0a_1$, provided the resulting form remains coercive.
\end{remark}

\begin{lemma}[Weighted norm equivalence]
\label[lemma]{lem:weighted-equivalence}
Let $X$ be a Hilbert space and let $\beta \geq 0$. For any strongly measurable function $g:(0,\Tend)\to X$, one has
\begin{align*}
  g\in L^2_\beta(0,\Tend;X)
  \quad\Longleftrightarrow\quad
  g\in L^2(0,\Tend;X).
\end{align*}
Furthermore,
\begin{align}
  e^{-\beta\Tend}\norm{g}_{L^2(0,\Tend;X)}
  \leq
  \norm{g}_{L^2_\beta(0,\Tend;X)}
  \leq
  \norm{g}_{L^2(0,\Tend;X)}.
  \label{eq:weighted-norm-equivalence}
\end{align}
\end{lemma}

\begin{proof}
This follows from $e^{-2\beta\Tend}\le e^{-2\beta t}\le1$ for $0<t<\Tend$; see~\cite[Lemma~3.1]{IshizakaMixed2026}.
\end{proof}

\begin{theorem}[Weighted well-posedness]\label[theorem]{thm:wellposedness}
Let $\mu\in\M^0([0,\tau])$ and choose $\beta\ge0$ such that
\begin{align*}
  q_{\mu,\beta}:=\frac{\Lambda_1\kappa_\mu(\beta)}{\alpha_0}<1.
\end{align*}
Then, \eqref{eq:problem} has a unique solution $u_\mu\in\W(0,\Tend)$, and
\begin{align}
\norm{u_\mu}_{L^2_\beta(0,\Tend;V)}
&\leq \frac{1}{1-q_{\mu,\beta}} \Bigl[\alpha_0^{-1/2}\norm{u_0}_H +\alpha_0^{-1}\norm{f}_{L^2_\beta(0,\Tend;V^*)} \notag \\
&\hspace{2cm} +\alpha_0^{-1}\Lambda_1\norm{\mu}_{\M([0,\tau])}
  \norm{\psi}_{L^2(-\tau,0;V)}\Bigr].\label{eq:weighted-u-bound}
\end{align}
For each fixed $\mu\in\M^0([0,\tau])$, such a weight exists.
\end{theorem}

\begin{proof}
For the proof, see~\cite[Theorem~3.6]{IshizakaMixed2026}.
\end{proof}

For two kernels $\mu,\nu\in\M^0([0,\tau])$, let $u_\mu,u_\nu$ denote the solutions with common data and set
\begin{align*}
  d_{\mu,\beta}:=\alpha_0-\Lambda_1\kappa_\mu(\beta),
  \quad R_{\mu,\nu}:=\K_{\mu-\nu}\tilde u_\nu.
\end{align*}

\begin{theorem}[Residual kernel-stability estimate]\label[theorem]{thm:residual-stability}
Let $d_{\mu,\beta}>0$.  Then,
\begin{align}
  \norm{u_\mu-u_\nu}_{L^2_\beta(0,\Tend;V)}
  &\leq d_{\mu,\beta}^{-1}
  \norm{R_{\mu,\nu}}_{L^2_\beta(0,\Tend;V^*)},\label{eq:residual-stability}\\
  \sup_{0\le t\le\Tend}e^{-\beta t}
  \norm{u_\mu(t)-u_\nu(t)}_H
  &\le \frac{\sqrt{\alpha_0}}{d_{\mu,\beta}}
  \norm{R_{\mu,\nu}}_{L^2_\beta(0,\Tend;V^*)},\\
  \norm{\partial_tu_\mu-\partial_tu_\nu}_{L^2_\beta(0,\Tend;V^*)}
  &\le \frac{\alpha_0+\Lambda_0}{d_{\mu,\beta}}
  \norm{R_{\mu,\nu}}_{L^2_\beta(0,\Tend;V^*)}.
\end{align}
Consequently,
\begin{align}
&\norm{u_\mu-u_\nu}_{\W(0,\Tend)} +\norm{u_\mu-u_\nu}_{\mathcal C([0,\Tend];H)} \notag\\
&\quad \leq e^{\beta\Tend}
  \frac{1+\sqrt{\alpha_0}+\alpha_0+\Lambda_0}{d_{\mu,\beta}}
  \norm{R_{\mu,\nu}}_{L^2_\beta(0,\Tend;V^*)}.\label{eq:residual-full-unweighted}
\end{align}
\end{theorem}

\begin{proof}
For the proof, see~\cite[Theorem~5.1]{IshizakaMixed2026}.
\end{proof}

\begin{lemma}[A common exponential weight under narrow convergence]\label[lemma]{lem:common-weight}
Let $\mu_n,\mu\in\M_+([0,\tau])$, assume $\mu_n\rightharpoonup\mu$ narrowly, and suppose $\mu(\{0\})=0$.  For any $d_*\in(0,\alpha_0)$, there exist $\beta\ge0$ and $N$ such that
\begin{align*}
  \alpha_0-\Lambda_1\kappa_{\mu_n}(\beta)\geq d_*,
  \quad n\geq N,
\end{align*}
and the same inequality holds for $\mu$.
\end{lemma}

\begin{proof}
For the proof, see~\cite[Lemma~5.6]{IshizakaMixed2026}.
\end{proof}

\begin{theorem}[Narrow kernel stability]\label[theorem]{thm:narrow-stability}
Let $\mu_n,\mu\in\M_+^0([0,\tau])$ with $\mu_n\rightharpoonup\mu$ narrowly.  For common data, the corresponding solutions satisfy
\begin{align*}
  u_{\mu_n}\longrightarrow u_\mu
  \quad\text{in }\W(0,\Tend)\cap\mathcal C([0,\Tend];H).
\end{align*}
\end{theorem}

\begin{proof}
For the proof, see~\cite[Theorem~5.8]{IshizakaMixed2026}.
\end{proof}

\subsection{Delayed diffusion realisation}\label{sec:pde-realisation}
Let $\Omega\subset\R^d$, $d\in\{1,2,3\}$, be bounded and Lipschitz, and set
\begin{align*}
  H=L^2(\Omega),\quad V=H_0^1(\Omega),\quad V^*=H^{-1}(\Omega),
\end{align*}
with $\|v\|_V=\|\nabla v\|_{L^2(\Omega;\R^d)}$.  Let $B_0,B_1\in L^\infty(\Omega;\R^{d\times d})$ and assume
\begin{align}\label{eq:B0-ellipticity}
  \xi^\top B_0(x)\xi\ge\lambda_0|\xi|^2
\end{align}
for any $\xi\in\R^d$ and almost every $x$.  We define
\begin{align}\label{eq:diffusion-forms}
  a_i(w,v):=\int_\Omega B_i(x)\nabla w\cdot\nabla v\,\dd x,
  \quad i\in\{0,1\}.
\end{align}

\begin{proposition}[Verification of the diffusion assumptions]\label[proposition]{prop:pde-assumptions}
The forms \eqref{eq:diffusion-forms} satisfy \eqref{eq:form-bounded}--\eqref{eq:a0-coercive}, where one may take
\begin{align*}
  \alpha_0=\lambda_0,
  \quad \Lambda_i=\norm{B_i}_{L^\infty(\Omega;\mathcal L(\R^d))}.
\end{align*}
\end{proposition}

\begin{proof}
For the proof, see~\cite[Proposition~4.1]{IshizakaMixed2026}.
\end{proof}

Consequently, \Cref{thm:wellposedness} applies to the delayed-diffusion problem generated by \eqref{eq:diffusion-forms} for any $\mu\in\M^0([0,\tau])$, $u_0\in L^2(\Omega)$, $f\in L^2(0,\Tend;H^{-1}(\Omega))$, and $\psi\in L^2(-\tau,0;H_0^1(\Omega))$; see also~\cite[Theorem~4.4]{IshizakaMixed2026}.

\section{Quantitative concentration rates}\label{sec:rates}
We combine transport couplings with time-translation estimates for one fixed reference trajectory.  This gives $W_1^\theta$ solution bounds and, in particular, rates for distributed-to-discrete concentration at a fixed positive delay.

\subsection{Translation moduli and transport plans}
The translation modulus records the temporal information needed when delay locations are perturbed.  It requires no pointwise continuity of the trajectory.  In particular, an endpoint jump of a zero extension does not obstruct strong continuity of translations in $L^2(\R;V)$, although it may change the decay rate of the modulus.

For $G\in L^2(\R;V)$, we define
\begin{align}\label{eq:translation-modulus}
  \omega_G(r)
  :=\sup_{\abs{h}\le r}
  \norm{G(\cdot-h)-G}_{L^2(\R;V)},
  \quad r\ge0.
\end{align}
The map $h\mapsto\norm{G(\cdot-h)-G}_{L^2(\R;V)}$ is continuous. Therefore, $\omega_G$ is finite, non-decreasing, and continuous, and thus Borel measurable.  Strong continuity of translations gives
\begin{align}\label{eq:translation-modulus-zero}
  \omega_G(r)\longrightarrow0
  \quad(r\downarrow0).
\end{align}
Furthermore,
\begin{align}\label{eq:translation-modulus-bound}
  0 \leq \omega_G(r)\le2\norm{G}_{L^2(\R;V)}, \quad r\ge0.
\end{align}

Let $\mu,\nu\in\M_+([0,\tau])$ have the same mass $M>0$, and let $\Pi(\mu,\nu)$ denote the set of non-negative finite Borel measures on $[0,\tau]^2$ whose first and second marginals are $\mu$ and $\nu$, respectively.  This set is non-empty; for example, $M^{-1}\mu\otimes\nu\in\Pi(\mu,\nu)$.  We use the equal-mass transport cost
\begin{align}\label{eq:W1}
  W_1(\mu,\nu)
  :=\inf_{\pi\in\Pi(\mu,\nu)}
  \int_{[0,\tau]^2}\abs{s-r}\,\dd\pi(s,r).
\end{align}
For probability measures, this is the standard $1$-Wasserstein distance.  For equal finite masses, it has the natural linear scaling in $M$; see \cite{AmbrosioGigliSavare2008,Villani2009}.  For $G\in L^2(\R;V)$, we set
\begin{align}\label{eq:transport-modulus-cost}
  \mathfrak C_G(\mu,\nu)
  :=\inf_{\pi\in\Pi(\mu,\nu)}
  \int_{[0,\tau]^2}
  \omega_G(\abs{s-r})\,\dd\pi(s,r).
\end{align}

\begin{lemma}[Basic properties of the transport-modulus cost]\label[lemma]{lem:transport-cost}
Let $G\in L^2(\R;V)$, and let $\mu,\nu\in\M_+([0,\tau])$ have the same finite mass $M>0$.  Then,
\begin{align*}
  0 \leq \mathfrak C_G(\mu,\nu)
  \leq 2M\norm{G}_{L^2(\R;V)}.
\end{align*}
Furthermore, $\mathfrak C_G(\mu,\nu)=\mathfrak C_G(\nu,\mu)$.  Let $\mu_n\in\M_+([0,\tau])$ have the same mass $M$, and suppose that there are plans $\pi_n\in\Pi(\mu_n,\mu)$ and numbers $\varepsilon_n\downarrow0$ such that
\begin{align*}
  \pi_n\bigl(\{(s,r):\abs{s-r}>\varepsilon_n\}\bigr)
  \longrightarrow0.
\end{align*}
Then, $\mathfrak C_G(\mu_n,\mu)\to0$.
\end{lemma}

\begin{proof}
The bound follows from \eqref{eq:translation-modulus-bound} and $\pi([0,\tau]^2)=M$.  Symmetry follows by transposing a transport plan.  For the final assertion,
\begin{align*}
  \int_{[0,\tau]^2}\omega_G(\abs{s-r})\,\dd\pi_n(s,r)
  &\leq M\omega_G(\varepsilon_n)
   +2\norm{G}_{L^2(\R;V)}
    \pi_n \left(\{\abs{s-r}>\varepsilon_n\}\right).
\end{align*}
Both terms tend to zero by \eqref{eq:translation-modulus-zero} and the assumed concentration of the plans.  Because $\pi_n\in\Pi(\mu_n,\mu)$,
\begin{align*}
  0\le \mathfrak C_G(\mu_n,\mu)
  \leq
  \int_{[0,\tau]^2}
  \omega_G(\abs{s-r})\,\dd\pi_n(s,r)
  \longrightarrow0,
\end{align*}
which proves the claim.
\end{proof}

\begin{proposition}[From Wasserstein displacement to the trajectory modulus]\label[proposition]{prop:transport-W1-bridge}
Let $G\in L^2(\R;V)$, and let $\mu,\nu\in\M_+([0,\tau])$ have the same mass $M>0$.  Then, for any $\delta>0$,
\begin{align}\label{eq:transport-W1-bridge}
  \mathfrak C_G(\mu,\nu)
  \leq
  M\,\omega_G(\delta)
  +\frac{2\norm{G}_{L^2(\R;V)}}{\delta}
   W_1(\mu,\nu).
\end{align}
Consequently, if $\mu_n\rightharpoonup\mu$ narrowly, all measures have the same mass $M$, and the delay interval is compact, then $\mathfrak C_G(\mu_n,\mu)\to0$.
\end{proposition}

\begin{proof}
Let $\pi\in\Pi(\mu,\nu)$.  We split the transport integral into $\{\abs{s-r}\le\delta\}$ and $\{\abs{s-r}>\delta\}$.  On the first set, the integrand is bounded by $\omega_G(\delta)$; on the second, it is bounded by $2\norm{G}_{L^2(\R;V)}$.  Furthermore,
\begin{align*}
  \pi\bigl(\{\abs{s-r}>\delta\}\bigr)
  \leq \frac1\delta
  \int_{[0,\tau]^2}\abs{s-r}\,\dd\pi(s,r).
\end{align*}
Therefore,
\begin{align*}
  \int_{[0,\tau]^2}\omega_G(\abs{s-r})\,\dd\pi(s,r)
  \leq M\omega_G(\delta)
  +\frac{2\norm{G}_{L^2(\R;V)}}{\delta}
   \int_{[0,\tau]^2}\abs{s-r}\,\dd\pi(s,r).
\end{align*}
Taking the infimum over $\pi\in\Pi(\mu,\nu)$ proves \eqref{eq:transport-W1-bridge}.  On the compact interval $[0,\tau]$, equal-mass narrow convergence is equivalent to convergence in $W_1$ after normalising the measures by their common mass $M$.  We first choose $\delta>0$ so that $M\omega_G(\delta)$ is sufficiently small and then let $n\to\infty$ in the second term of \eqref{eq:transport-W1-bridge}.  This gives $\mathfrak C_G(\mu_n,\mu)\to0$.
\end{proof}

\begin{remark}[Why the transport modulus is more informative than narrow convergence]
Estimate \eqref{eq:transport-W1-bridge} recovers qualitative narrow stability while retaining the actual translation modulus of the reference trajectory.  When no algebraic time regularity is known, this modulus still provides quantitative information, although no uniform power-law estimate in the transport distance need be available.
\end{remark}

We now turn the displacement of delay mass into a solution estimate.  A transport plan pairs a delay location $s$ for $\mu$ with a delay location $r$ for $\nu$.  The corresponding residual is controlled by the translation difference $G(\cdot-s)-G(\cdot-r)$.  Integration over the plan and minimisation over all plans give $\mathfrak C_G(\mu,\nu)$.  The estimate remains asymmetric: the stability margin is determined by $\mu$, whereas the translation modulus is evaluated along the reference trajectory $\tilde u_\nu$.

\begin{theorem}[Transport-modulus kernel stability]\label[theorem]{thm:transport-modulus}
Let $\mu,\nu\in\M_+^0([0,\tau])$ have the same mass $M>0$, and choose $\beta\ge0$ with $d_{\mu,\beta}=\alpha_0-\Lambda_1\kappa_\mu(\beta)>0$.  Assume that the joined reference trajectory $\tilde u_\nu$ admits an extension $G\in L^2(\R;V)$.  Then,
\begin{align}\label{eq:transport-modulus-stability}
  \norm{u_\mu-u_\nu}_{L^2_\beta(0,\Tend;V)}
  \leq
  \frac{\Lambda_1}{d_{\mu,\beta}}
  \mathfrak C_G(\mu,\nu).
\end{align}
Furthermore,
\begin{align}
&\norm{u_\mu-u_\nu}_{\W(0,\Tend)} +\norm{u_\mu-u_\nu}_{\mathcal C([0,\Tend];H)} \notag \\
&\quad \leq e^{\beta\Tend} \frac{\Lambda_1(1+\sqrt{\alpha_0}+\alpha_0+\Lambda_0)}{d_{\mu,\beta}} \mathfrak C_G(\mu,\nu). \label{eq:transport-modulus-full}
\end{align}
If, for some $\varepsilon\ge0$, a coupling $\pi\in\Pi(\mu,\nu)$ is supported in $\{(s,r):\abs{s-r}\le\varepsilon\}$, then
\begin{align}\label{eq:transport-modulus-epsilon}
\norm{u_\mu-u_\nu}_{L^2_\beta(0,\Tend;V)}
\leq \frac{\Lambda_1M}{d_{\mu,\beta}} \omega_G(\varepsilon).
\end{align}
The corresponding unweighted $\W(0,\Tend)\cap \mathcal{C}([0,\Tend];H)$ estimate follows from \eqref{eq:transport-modulus-full} with $\mathfrak C_G(\mu,\nu)\le M\omega_G(\varepsilon)$.
\end{theorem}

\begin{proof}
For $s\in[0,\tau]$, define
\begin{align*}
F(s):=A_1G(\cdot-s)|_{(0,\Tend)} \in Y:=L^2(0,\Tend;V^*).
\end{align*}
The translation group is strongly continuous on $L^2(\R;V)$, and hence
\begin{align*}
\norm{F(s)-F(r)}_Y \leq \Lambda_1 \norm{G(\cdot-s)-G(\cdot-r)}_{L^2(\R;V)}.
\end{align*}
Thus, $F:[0,\tau]\to Y$ is continuous and Bochner integrable against any finite measure considered below.  For any $\pi\in\Pi(\mu,\nu)$, the marginal identities and Bochner--Fubini give
\begin{align}\label{eq:transport-residual-identity}
  \K_{\mu-\nu}\tilde u_\nu
  =\int_{[0,\tau]^2}\bigl(F(s)-F(r)\bigr)\,\dd\pi(s,r)
  \quad\text{in }Y.
\end{align}
Minkowski's inequality, boundedness of $A_1$, and the definition of $\omega_G$ yield
\begin{align*}
  \norm{\K_{\mu-\nu}\tilde u_\nu}_{Y}
  &\leq \Lambda_1
  \int_{[0,\tau]^2}
  \norm{G(\cdot-s)-G(\cdot-r)}_{L^2(\R;V)}
  \,\dd\pi(s,r)\\
  &\leq \Lambda_1
  \int_{[0,\tau]^2}
  \omega_G(\abs{s-r})\,\dd\pi(s,r).
\end{align*}
Because the weighted residual norm is bounded by the unweighted one, we take the infimum over $\pi\in\Pi(\mu,\nu)$ and apply \eqref{eq:residual-stability} to obtain \eqref{eq:transport-modulus-stability}.  The same residual bound, combined with \eqref{eq:residual-full-unweighted}, proves \eqref{eq:transport-modulus-full}.  Finally, if $\pi$ is supported in $\{(s,r):\abs{s-r}\le\varepsilon\}$, then
\begin{align*}
  \mathfrak C_G(\mu,\nu)
  \le
  \int_{[0,\tau]^2}
  \omega_G(\abs{s-r})\,\dd\pi(s,r)
  \le M\omega_G(\varepsilon).
\end{align*}
Substitution into \eqref{eq:transport-modulus-stability} and \eqref{eq:transport-modulus-full} proves the supported-coupling assertions.
\end{proof}

\subsection{Fractional regularity and Wasserstein rates}
For $0<\theta\le1$, let $H^\theta(\R;V)$ be the usual Bessel-potential space.  For an interval $I$, we use the restriction space $H^\theta(I;V)$, whose norm is the infimum of the $H^\theta(\R;V)$-norms over all extensions.  Thus, the phrase ``admits an $H^\theta$-extension'' has its usual restriction-space meaning.

\begin{lemma}[Fractional translation estimate]\label[lemma]{lem:translation}
There exists a constant $C_\theta>0$, depending only on $\theta$, such that, for any $G\in H^\theta(\R;V)$ and any $r,s\in\R$,
\begin{align}\label{eq:translation-estimate}
  \norm{G(\cdot-s)-G(\cdot-r)}_{L^2(\R;V)}
  \leq C_\theta\abs{s-r}^\theta
  \norm{G}_{H^\theta(\R;V)}.
\end{align}
\end{lemma}

\begin{proof}
Let $\widehat G$ denote the Fourier transform in time, with $V$ complexified in the usual way.  By Plancherel's theorem,
\begin{align*}
  \norm{G(\cdot-s)-G(\cdot-r)}_{L^2(\R;V)}^2
  &=\frac{1}{2\pi}\int_\R
  \abs{e^{-i\xi s}-e^{-i\xi r}}^2
  \norm{\widehat G(\xi)}_V^2\,\dd\xi.
\end{align*}
For $0<\theta\le1$,
\begin{align*}
  \abs{e^{-iz}-1}
  \leq \min\{2,\abs z\}
  \leq 2^{1-\theta}\abs z^\theta.
\end{align*}
Consequently,
\begin{align*}
  \abs{e^{-i\xi s}-e^{-i\xi r}}
  &\le 2^{1-\theta}\abs{\xi}^\theta
  \abs{s-r}^\theta.
\end{align*}
Substitution and
$\abs{\xi}^{2\theta}\le(1+\abs{\xi}^2)^\theta$ give
\begin{align*}
  \norm{G(\cdot-s)-G(\cdot-r)}_{L^2(\R;V)}^2
  &\le
  2^{2(1-\theta)}\abs{s-r}^{2\theta}
  \norm{G}_{H^\theta(\R;V)}^2.
\end{align*}
Taking square roots proves \eqref{eq:translation-estimate}.
\end{proof}

\begin{proposition}[Gluing criterion for the joined trajectory]\label[proposition]{prop:Htheta-gluing}
Let $0<\theta\le1$, let $\psi\in H^\theta(-\tau,0;V)$ and $u\in H^\theta(0,\Tend;V)$, and let $\tilde u$ be their joined function from \eqref{eq:joined-trajectory}.
\begin{enumerate}
\item If $0<\theta<1/2$, then $\tilde u\in H^\theta(-\tau,\Tend;V)$ without a trace compatibility condition, and
\begin{align}\label{eq:Htheta-gluing-low}
  \norm{\tilde u}_{H^\theta(-\tau,\Tend;V)}
  \leq C \left(
  \norm{\psi}_{H^\theta(-\tau,0;V)}
  +\norm{u}_{H^\theta(0,\Tend;V)}\right).
\end{align}
\item If $1/2<\theta\le1$, then $\tilde u\in H^\theta(-\tau,\Tend;V)$ if and only if the one-sided $V$-traces agree:
\begin{align}\label{eq:Htheta-trace-match}
  \psi(0-)=u(0+)
  \quad \text{in }V.
\end{align}
When \eqref{eq:Htheta-trace-match} holds, the analogue of \eqref{eq:Htheta-gluing-low} is valid.
\item At $\theta=1/2$, the precise additional requirement is
\begin{align}\label{eq:Hhalf-cross-condition}
  \int_{-\tau}^{0}\int_0^\Tend
  \frac{\norm{u(t)-\psi(s)}_V^2}{(t-s)^2}
  \,\dd t\,\dd s<\infty.
\end{align}
Together with the separate $H^{1/2}$-regularity of $u$ and $\psi$, condition \eqref{eq:Hhalf-cross-condition} is equivalent to $\tilde u\in H^{1/2}(-\tau,\Tend;V)$.
\end{enumerate}
In any case, membership on the bounded interval implies the existence of an extension $G\in H^\theta(\R;V)$, and the extension may be chosen with norm bounded by a constant times the interval norm.
\end{proposition}

\begin{proof}
For $0<\theta<1$, the Bessel-potential norm on a bounded interval is equivalent to the Hilbert-valued Slobodeckij norm.  We use the standard one-dimensional zero-extension, trace, and fractional Hardy results; see \cite{DiNezzaPalatucciValdinoci2012} and \cite[Theorems~3.33, 3.37, and 3.40]{McLean2000}.  Their Hilbert-valued versions follow by applying the scalar estimates to the coefficients with respect to an orthonormal basis of $V$ and then summing.  The Slobodeckij seminorm of the joined function splits into the two same-side seminorms and twice the cross-interface integral
\begin{align}\label{eq:cross-general}
  \int_{-\tau}^{0}\int_0^\Tend
  \frac{\norm{u(t)-\psi(s)}_V^2}
       {(t-s)^{1+2\theta}}
  \,\dd t\,\dd s.
\end{align}

Suppose first that $0<\theta<1/2$.  Extension by zero across an endpoint is bounded on $H^\theta$.  Extend $\psi$ by zero to $(0,\Tend)$ and extend $u$ by zero to $(-\tau,0)$.  The sum of these two extensions is $\tilde u$, and hence
\begin{align*}
  \norm{\tilde u}_{H^\theta(-\tau,\Tend;V)}
  \leq C\left(
  \norm{\psi}_{H^\theta(-\tau,0;V)}
  +\norm{u}_{H^\theta(0,\Tend;V)}
  \right).
\end{align*}
This proves the first assertion.

Let $1/2<\theta<1$.  The one-sided $V$-traces of $\psi$ and $u$ exist.  Suppose first that they agree, and denote their common value by $a\in V$.  Then,
\begin{align*}
  \norm{u(t)-\psi(s)}_V^2
  \leq
  2\norm{u(t)-a}_V^2
  +2\norm{\psi(s)-a}_V^2.
\end{align*}
Furthermore,
\begin{align*}
  \int_{-\tau}^{0}
  \frac{\dd s}{(t-s)^{1+2\theta}}
  &\le \frac{1}{2\theta}t^{-2\theta},
  \quad t>0,\\
  \int_0^\Tend
  \frac{\dd t}{(t-s)^{1+2\theta}}
  &\le \frac{1}{2\theta}\abs{s}^{-2\theta},
  \quad s<0.
\end{align*}
Consequently, \eqref{eq:cross-general} is bounded by a constant times
\begin{align*}
  \int_0^\Tend
  t^{-2\theta}\norm{u(t)-a}_V^2\,\dd t
  +
  \int_{-\tau}^{0}
  \abs{s}^{-2\theta}\norm{\psi(s)-a}_V^2\,\dd s.
\end{align*}
The fractional Hardy inequality for functions with vanishing trace gives
\begin{align*}
  \int_0^\Tend
  t^{-2\theta}\norm{u(t)-a}_V^2\,\dd t
  &\le
  C\norm{u}_{H^\theta(0,\Tend;V)}^2,\\
  \int_{-\tau}^{0}
  \abs{s}^{-2\theta}\norm{\psi(s)-a}_V^2\,\dd s
  &\le
  C\norm{\psi}_{H^\theta(-\tau,0;V)}^2.
\end{align*}
Thus, the cross-interface integral is finite, and the analogue of \eqref{eq:Htheta-gluing-low} follows.  Conversely, if $\tilde u\in H^\theta(-\tau,\Tend;V)$, then its $V$-trace at $t=0$ is well defined.  The one-sided traces of its restrictions to $(-\tau,0)$ and $(0,\Tend)$ must both equal this trace.  Therefore,
\begin{align*}
  \psi(0-)=u(0+)
  \quad\text{in }V.
\end{align*}
This proves the second assertion for $1/2<\theta<1$.

For $\theta=1$, the conclusion is the standard $H^1$ gluing result.  Indeed, the distributional derivative of the piecewise $H^1$ function contains an interface Dirac mass precisely when the two traces differ.  Thus, $\tilde u\in H^1(-\tau,\Tend;V)$ if and only if the traces agree; in that case, the corresponding norm estimate follows from the piecewise $H^1$ estimates.

Let $\theta=1/2$.  In the decomposition of the Slobodeckij seminorm, the two same-side terms are finite by the assumed separate $H^{1/2}$-regularity.  The remaining term is twice \eqref{eq:Hhalf-cross-condition}.  Therefore, \eqref{eq:Hhalf-cross-condition} is finite if and only if
\begin{align*}
  \tilde u\in H^{1/2}(-\tau,\Tend;V),
\end{align*}
which proves the third assertion.

Bounded intervals admit bounded extension operators for Hilbert-valued $H^\theta$ spaces.  Therefore, in any case, $\tilde u$ admits an extension $G\in H^\theta(\R;V)$ satisfying
\begin{align*}
  \norm{G}_{H^\theta(\R;V)}
  \le
  C\norm{\tilde u}_{H^\theta(-\tau,\Tend;V)}.
\end{align*}
All constants above depend only on $\theta$, $\tau$, and $\Tend$.
\end{proof}

\begin{lemma}[Differentiation of a measure-valued history]\label[lemma]{lem:history-differentiation}
Let $\sigma\in\M([0,\tau])$ and let $G\in H^1(-\tau,\Tend;V)$.  Then,
\begin{align*}
  \K_\sigma G\in H^1(0,\Tend;V^*)
\end{align*}
and
\begin{align}\label{eq:history-derivative}
  \partial_t(\K_\sigma G)
  =\K_\sigma(\partial_tG)
  \quad \text{in }L^2(0,\Tend;V^*).
\end{align}
Furthermore,
\begin{align}\label{eq:history-H1-bound}
  \norm{\K_\sigma G}_{H^1(0,\Tend;V^*)}
  \le
  \Lambda_1\norm{\sigma}_{\M([0,\tau])}
  \norm{G}_{H^1(-\tau,\Tend;V)}.
\end{align}
\end{lemma}

\begin{proof}
The $L^2(0,\Tend;V^*)$ bound follows from \Cref{lem:measure-convolution}.  Let $v\in V$ and
$\varphi\in C_c^\infty(0,\Tend)$.  The boundedness of $a_1$, the finite variation of $\sigma$, and Cauchy--Schwarz justify Fubini's theorem.  Indeed,
\begin{align*}
  &\int_{[0,\tau]}\int_0^\Tend
  \abs{a_1(G(t-s),v)\varphi'(t)}
  \,\dd t\,\dd\abs{\sigma}(s)\\
  &\quad\le
  \Lambda_1\norm{v}_V\norm{\sigma}_{\M([0,\tau])}
  \norm{G}_{L^2(-\tau,\Tend;V)}
  \norm{\varphi'}_{L^2(0,\Tend)}<\infty.
\end{align*}
Because $t\mapsto G(t-s)$ belongs to $H^1(0,\Tend;V)$ for any fixed $s\in[0,\tau]$, integration by parts gives
\begin{align*}
 -\int_0^\Tend
  \pair{(\K_\sigma G)(t)}{v}\varphi'(t)\,\dd t
 &= -\int_{[0,\tau]}\int_0^\Tend
    a_1(G(t-s),v)\varphi'(t)\,\dd t\,\dd\sigma(s)\\
 &= \int_{[0,\tau]}\int_0^\Tend
    a_1(\partial_tG(t-s),v)\varphi(t)\,\dd t\,\dd\sigma(s)\\
 &= \int_0^\Tend
    \pair{(\K_\sigma\partial_tG)(t)}{v}\varphi(t)\,\dd t.
\end{align*}
There are no boundary terms because $\varphi$ is compactly supported in $(0,\Tend)$.  Since $v\in V$ is arbitrary and $\K_\sigma(\partial_tG)\in L^2(0,\Tend;V^*)$ by \Cref{lem:measure-convolution}, this proves \eqref{eq:history-derivative} as an identity of $V^*$-valued distributions.  Consequently, $\K_\sigma G\in H^1(0,\Tend;V^*)$.  Finally,
\begin{align*}
  \norm{\K_\sigma G}_{H^1(0,\Tend;V^*)}^2
  &=
  \norm{\K_\sigma G}_{L^2(0,\Tend;V^*)}^2
  +
  \norm{\K_\sigma(\partial_tG)}_{L^2(0,\Tend;V^*)}^2\\
  &\le
  \Lambda_1^2\norm{\sigma}_{\M([0,\tau])}^2
  \norm{G}_{H^1(-\tau,\Tend;V)}^2.
\end{align*}
Taking square roots proves \eqref{eq:history-H1-bound}.
\end{proof}

\begin{proposition}[A data criterion for an $H^1$ joined trajectory]\label[proposition]{prop:data-H1-criterion}
Let $\mu\in\M^0([0,\tau])$, and assume
\begin{align}\label{eq:H1-data}
\begin{gathered}
  u_0\in V,
  \quad
  f\in H^1(0,\Tend;V^*),
  \quad
  \psi\in H^1(-\tau,0;V),\\
  \psi(0-)=u_0\quad\text{in }V.
\end{gathered}
\end{align}
Because
\begin{align*}
H^1(0,\Tend;V^*)\hookrightarrow\mathcal C([0,\Tend];V^*), \quad H^1(-\tau,0;V)\hookrightarrow\mathcal C([-\tau,0];V),	
\end{align*}
the values $f(0)$ and $\psi(-s)$ are well defined, and we set
\begin{align}\label{eq:initial-derivative}
  w_0
  :=f(0)-A_0u_0
    -\int_{[0,\tau]}A_1\psi(-s)\,\dd\mu(s)
  \in V^*.
\end{align}
Suppose, in addition, that $w_0\in H$.  Then, the solution $u_\mu$ of \eqref{eq:problem} satisfies
\begin{align}\label{eq:H1-solution-regularity}
  u_\mu\in H^1(0,\Tend;V),
  \quad
  \partial_tu_\mu\in\W(0,\Tend).
\end{align}
Consequently, after assigning the common trace $u_0$ at $t=0$,
\begin{align*}
  \tilde u_\mu\in H^1(-\tau,\Tend;V),
\end{align*}
and the joined trajectory admits an extension $G\in H^1(\R;V)$.  More precisely, for any $\beta\ge0$ with $q_{\mu,\beta}<1$,
\begin{align}\label{eq:time-derivative-bound}
  \norm{\partial_tu_\mu}_{L^2_\beta(0,\Tend;V)}
  \le \frac{1}{1-q_{\mu,\beta}}
  \Biggl(&\alpha_0^{-1/2}\norm{w_0}_H
  +\alpha_0^{-1}\norm{\partial_tf}_{L^2_\beta(0,\Tend;V^*)}\notag\\
  &+\alpha_0^{-1}\Lambda_1\norm{\mu}_{\M([0,\tau])}
  \norm{\partial_t\psi}_{L^2(-\tau,0;V)}\Biggr).
\end{align}
Together with \eqref{eq:weighted-u-bound}, the norm equivalence of \Cref{lem:weighted-equivalence}, and a bounded extension operator, this bounds the norm $\norm{G}_{H^1(\R;V)}$ in terms of the data, $\mu$, $\beta$, $\tau$, $\Tend$, and the
structural constants.
\end{proposition}

\begin{proof}
Choose $\beta\ge0$ such that $q_{\mu,\beta}<1$; such a choice exists by \Cref{thm:wellposedness}.  We apply \Cref{thm:wellposedness} to the differentiated data $(w_0,\partial_tf,\partial_t\psi)$.  This gives a unique $w\in\W(0,\Tend)$ satisfying
\begin{align}\label{eq:differentiated-problem}
  \partial_tw+A_0w+\K_\mu\widetilde w=\partial_tf,
  \quad
  w(0)=w_0,
\end{align}
where $\widetilde w=\partial_t\psi$ on $(-\tau,0)$ and $\widetilde w=w$ on $(0,\Tend)$.  Estimate
\eqref{eq:time-derivative-bound} is \eqref{eq:weighted-u-bound} applied to \eqref{eq:differentiated-problem}.

We define
\begin{align*}
U(t):=u_0+\int_0^t w(r)\,\dd r.
\end{align*}
Because $w\in L^2(0,\Tend;V)$, $U\in H^1(0,\Tend;V)$ and $\partial_tU=w$.  Let $\tilde u$ be the function obtained by joining $U$ to $\psi$ on $(-\tau,0)$.  The trace compatibility in \eqref{eq:H1-data} gives
\begin{align}\label{eq:joined-U-H1}
  \tilde u\in H^1(-\tau,\Tend;V),
  \quad
  \partial_t\tilde u = \tilde w.
\end{align}
By \Cref{lem:history-differentiation}, $\K_\mu\tilde u\in H^1(0,\Tend;V^*)$ and
\begin{align*}
  \partial_t(\K_\mu\tilde u)
  =\K_\mu \tilde w.
\end{align*}
Because $w\in\W(0,\Tend)$ and $U\in H^1(0,\Tend;V)$, both $w$ and $A_0U$ belong to $H^1(0,\Tend;V^*)$.  The same holds for $\K_\mu\tilde u$ and $f$.  Thus, the function
\begin{align*}
  \mathcal F
  :=w+A_0U+\K_\mu\tilde u-f
\end{align*}
belongs to $H^1(0,\Tend;V^*)$.  Using \eqref{eq:differentiated-problem} and \eqref{eq:joined-U-H1}, we obtain
\begin{align*}
  \partial_t\mathcal F
  =\partial_tw+A_0w+\K_\mu \tilde w-\partial_tf
  =0
\end{align*}
in $L^2(0,\Tend;V^*)$.  The function $\tilde u$ is continuous on the compact interval $[-\tau,\Tend]$ with values in $V$.  Therefore, finite variation of $\mu$ and dominated convergence imply that $\K_\mu\tilde u$ is continuous with values in $V^*$ and that
\begin{align*}
  (\K_\mu\tilde u)(0)
  =\int_{[0,\tau]}A_1\psi(-s)\,\dd\mu(s).
\end{align*}
Because $U(0)=u_0$ and $w(0)=w_0$ in $H\hookrightarrow V^*$,
\begin{align*}
  \mathcal F(0)
  &=
  w_0+A_0u_0
  +\int_{[0,\tau]}A_1\psi(-s)\,\dd\mu(s)-f(0)
  =0
\end{align*}
by \eqref{eq:initial-derivative}.  Thus, $\mathcal F=0$ on $[0,\Tend]$, so $U$ solves \eqref{eq:problem}.  Uniqueness in \Cref{thm:wellposedness} yields $U=u_\mu$.  Consequently, $\partial_tu_\mu=w\in\W(0,\Tend)$ and \eqref{eq:H1-solution-regularity} follows.

Finally, the standard bounded extension operator on an interval gives an extension of the joined trajectory in $H^1(\R;V)$.  The asserted norm dependence follows from \eqref{eq:weighted-u-bound}, \eqref{eq:time-derivative-bound}, and \Cref{lem:weighted-equivalence}.
\end{proof}

\begin{remark}[How to verify the rate assumption in practice]\label[remark]{rem:Htheta-extension}
The rate theorem requires regularity of the \emph{joined} trajectory, not separate estimates for the past and future pieces.  \Cref{prop:Htheta-gluing} gives the exact gluing test, while \Cref{prop:data-H1-criterion} provides a sufficient condition stated directly in terms of the equation data.  In particular, a jump at $t=0$ is permitted when $0<\theta<1/2$, whereas for $\theta>1/2$ the history and solution must have the same $V$-trace.  The borderline case $\theta=1/2$ is governed by \eqref{eq:Hhalf-cross-condition} and cannot be reduced to a pointwise matching condition.  By contrast, the qualitative transport-modulus theorem requires only an $L^2(\R;V)$ extension.  The zero extension is always available at that level, so no interface compatibility is imposed.
\end{remark}

Recall the equal-mass transport cost $W_1$ from \eqref{eq:W1}.

\begin{theorem}[Wasserstein kernel-stability rate]\label[theorem]{thm:wasserstein-rate}
Let $\mu,\nu\in\M_+^0([0,\tau])$ have the same mass $M>0$, and choose $\beta\ge0$ with $d_{\mu,\beta}>0$.  Assume that the joined reference trajectory $\tilde u_\nu$ admits an extension $G\in H^\theta(\R;V)$ for some $0<\theta\le1$.  Then,
\begin{align}\label{eq:wasserstein-rate}
  \norm{u_\mu-u_\nu}_{L^2_\beta(0,\Tend;V)}
  \le
  \frac{C_\theta\Lambda_1M^{1-\theta}}
       {d_{\mu,\beta}}
  W_1(\mu,\nu)^\theta
  \norm{G}_{H^\theta(\R;V)}.
\end{align}
Furthermore,
\begin{align}\label{eq:wasserstein-rate-full}
  &\norm{u_\mu-u_\nu}_{\W(0,\Tend)}
  +\norm{u_\mu-u_\nu}_{\mathcal{C}([0,\Tend];H)}
  \\
  &\quad\le
  e^{\beta\Tend}
  \frac{C_\theta\Lambda_1(1+\sqrt{\alpha_0}+\alpha_0+\Lambda_0)
        M^{1-\theta}}
       {d_{\mu,\beta}}
  W_1(\mu,\nu)^\theta
  \norm{G}_{H^\theta(\R;V)}.
  \notag
\end{align}
\end{theorem}

\begin{proof}
We use the continuous $Y$-valued map $F$ from the proof of \Cref{thm:transport-modulus}. For any $\pi\in\Pi(\mu,\nu)$, identity \eqref{eq:transport-residual-identity}, Minkowski's inequality, boundedness of $A_1$, and \Cref{lem:translation} give
\begin{align*}
  \norm{\K_{\mu-\nu}\tilde u_\nu}_{L^2(0,\Tend;V^*)}
  &\leq
  C_\theta\Lambda_1
  \norm{G}_{H^\theta(\R;V)}
  \int_{[0,\tau]^2}
  \abs{s-r}^\theta\,\dd\pi(s,r).
\end{align*}
Because $\pi([0,\tau]^2)=M$, Jensen's inequality for the probability measure $\pi/M$ yields
\begin{align*}
  \int_{[0,\tau]^2}
  \abs{s-r}^\theta\,\dd\pi(s,r)
  &\leq
  M^{1-\theta}
  \left(
    \int_{[0,\tau]^2}
    \abs{s-r}\,\dd\pi(s,r)
  \right)^\theta.
\end{align*}
Consequently, for any $\pi\in\Pi(\mu,\nu)$,
\begin{align*}
  \norm{\K_{\mu-\nu}\tilde u_\nu}_{L^2(0,\Tend;V^*)}
  &\le
  C_\theta\Lambda_1M^{1-\theta}
  \norm{G}_{H^\theta(\R;V)}
  \left(
    \int_{[0,\tau]^2}
    \abs{s-r}\,\dd\pi(s,r)
  \right)^\theta.
\end{align*}
Because $x\mapsto x^\theta$ is non-decreasing on $[0,\infty)$, taking the infimum over $\pi\in\Pi(\mu,\nu)$ and using \eqref{eq:W1} give
\begin{align*}
  \norm{\K_{\mu-\nu}\tilde u_\nu}_{L^2(0,\Tend;V^*)}
  &\leq
  C_\theta\Lambda_1M^{1-\theta}
  W_1(\mu,\nu)^\theta
  \norm{G}_{H^\theta(\R;V)}.
\end{align*}
Because $\beta\ge0$, the weighted residual norm is bounded by the unweighted one.  Applying \eqref{eq:residual-stability} and \eqref{eq:residual-full-unweighted} proves \eqref{eq:wasserstein-rate} and \eqref{eq:wasserstein-rate-full}, respectively.
\end{proof}

\begin{remark}[A data-driven Lipschitz rate for delayed diffusion]
\label[remark]{rem:pde-data-W1}
In the delayed-diffusion setting of \Cref{sec:pde-realisation}, $\pair{A_iz}{v}=\pair{-\nabla\!\cdot(B_i\nabla z)}{v}$ for $z,v\in H_0^1(\Omega)$.  Let $\nu\in\M_+^0([0,\tau])$, let $u_0\in H_0^1(\Omega)$, $f\in H^1(0,\Tend;H^{-1}(\Omega))$, and $\psi\in H^1(-\tau,0;H_0^1(\Omega))$ with $\psi(0-)=u_0$, and suppose that
\begin{align}\label{eq:pde-initial-derivative}
w_0:=f(0) +\nabla\!\cdot(B_0\nabla u_0) +\int_{[0,\tau]} \nabla\!\cdot\left(B_1\nabla\psi(-s)\right)\,\dd\nu(s),
\end{align}
understood a priori in $H^{-1}(\Omega)$, belongs to $L^2(\Omega)$.  Then, \Cref{prop:pde-assumptions} and \Cref{prop:data-H1-criterion} provide an extension $G\in H^1(\R;H_0^1(\Omega))$ of the joined reference trajectory, and \eqref{eq:wasserstein-rate-full} with $\theta=1$ gives
\begin{align*}
  \norm{u_\mu-u_\nu}_{\W(0,\Tend)}
  +\norm{u_\mu-u_\nu}_{\mathcal{C}([0,\Tend];L^2(\Omega))}
  \le C\,W_1(\mu,\nu)
\end{align*}
for any $\mu\in\M_+^0([0,\tau])$ of the same mass as $\nu$, provided that $\beta\ge0$ is chosen with $\lambda_0-\norm{B_1}_{L^\infty(\Omega;\mathcal L(\R^d))}\kappa_\mu(\beta)>0$. Here, $C$ depends on $\lambda_0$, the coefficient bounds, $\Tend$, this chosen $\beta$ and its positive stability margin, the common mass, and $\norm{G}_{H^1(\R;H_0^1(\Omega))}$.
\end{remark}


\subsection{Distributed-to-discrete concentration}
We apply the transport estimates to three standard kernel approximations: concentration of a distributed kernel to a single delay, motion of finitely many delay atoms, and finite atomic quadrature of a distributed law.  The key point is separation from the origin.  At a fixed positive lag, one exponential weight provides a common stability margin for all sufficiently concentrated kernels, while explicit transport plans determine the error scale.  We begin with an absolutely continuous approximation of a single discrete delay.

\begin{corollary}[Concentration to a discrete delay]\label[corollary]{cor:delay-concentration}
Let $\tau_d\in(0,\tau)$, $m>0$, and let $\eta\in L^1(-1,1)$ satisfy $\eta\ge0$ and $\int_{-1}^1\eta(z)\,\dd z=1$.  Extend $\eta$ by zero outside $(-1,1)$.  For $0<\varepsilon<\min\{\tau_d,\tau-\tau_d\}$, we define
\begin{align}\label{eq:mollified-delay}
  \dd\mu_\varepsilon(s)
  :=\frac{m}{\varepsilon}
  \eta\left(\frac{s-\tau_d}{\varepsilon}\right)\dd s,
  \quad
  \nu:=m\delta_{\tau_d}.
\end{align}
Let all solutions below correspond to the same data $(u_0,f,\psi)$.  Then, $\mu_\varepsilon\rightharpoonup\nu$ narrowly.  Let $G_0\in L^2(\R;V)$ be the zero extension of the joined delay trajectory $\tilde u_\nu$.  For any fixed $0<\varepsilon_0<\min\{\tau_d,\tau-\tau_d\}$ there is a constant $C_{\varepsilon_0}$, independent of $0<\varepsilon\le\varepsilon_0$, such that
\begin{align}\label{eq:delay-modulus-rate}
  \norm{u_{\mu_\varepsilon}-u_\nu}_{\W(0,\Tend)}
  +\norm{u_{\mu_\varepsilon}-u_\nu}_{\mathcal{C}([0,\Tend];H)}
  \le C_{\varepsilon_0}\,\omega_{G_0}(\varepsilon).
\end{align}
Consequently,
\begin{align}\label{eq:qualitative-delay-convergence}
  u_{\mu_\varepsilon}\longrightarrow u_\nu
  \quad\text{in }\W(0,\Tend) \cap \mathcal{C}([0,\Tend];H).
\end{align}
If $\tilde u_\nu$ admits an extension $G_\theta\in H^\theta(\R;V)$ for some
$0<\theta\le1$, then
\begin{align}\label{eq:delay-rate}
  \norm{u_{\mu_\varepsilon}-u_\nu}_{\W(0,\Tend)}
  +\norm{u_{\mu_\varepsilon}-u_\nu}_{\mathcal{C}([0,\Tend];H)}
  \le C_{\varepsilon_0,\theta}\varepsilon^\theta.
\end{align}
If the reference kernel $\nu$ and the common data satisfy the hypotheses of \Cref{prop:data-H1-criterion}, then \eqref{eq:delay-rate} holds with $\theta=1$.
\end{corollary}

\begin{proof}
For $0<\varepsilon<\min\{\tau_d,\tau-\tau_d\}$, the measures $\mu_\varepsilon$ and $\nu$ belong to $\M_+^0([0,\tau])$ and have mass $m$. The change of variables $s=\tau_d+\varepsilon z$ and dominated convergence give $\mu_\varepsilon\rightharpoonup\nu$ narrowly.  The plan $\pi_\varepsilon:=\mu_\varepsilon\otimes\delta_{\tau_d}$ belongs to $\Pi(\mu_\varepsilon,\nu)$ and is supported in $\{(s,r):\abs{s-r}\le\varepsilon\}$.  For $0<\varepsilon\le\varepsilon_0$ and $\beta\ge0$,
\begin{align*}
  \kappa_{\mu_\varepsilon}(\beta)
  \le m e^{-\beta(\tau_d-\varepsilon_0)}.
\end{align*}
Choosing one $\beta\ge0$ with $\Lambda_1me^{-\beta(\tau_d-\varepsilon_0)}<\alpha_0/2$ therefore gives $d_{\mu_\varepsilon,\beta}\ge\alpha_0/2$ for all such $\varepsilon$.  Then, \eqref{eq:transport-modulus-full} with the extension $G_0$ and the coupling $\pi_\varepsilon$ yields \eqref{eq:delay-modulus-rate} with
\begin{align*}
  C_{\varepsilon_0}
  =2e^{\beta\Tend}\Lambda_1m(1+\sqrt{\alpha_0}+\alpha_0+\Lambda_0)/\alpha_0,
\end{align*}
and \eqref{eq:qualitative-delay-convergence} follows from \eqref{eq:translation-modulus-zero}.  If $G_\theta\in H^\theta(\R;V)$ extends $\tilde u_\nu$, \Cref{lem:translation} gives $\omega_{G_\theta}(\varepsilon)\le C_\theta\varepsilon^\theta\norm{G_\theta}_{H^\theta(\R;V)}$, and the same argument with $G_\theta$ proves \eqref{eq:delay-rate}.  Under the hypotheses of \Cref{prop:data-H1-criterion}, that proposition provides such an extension with $\theta=1$.
\end{proof}

\begin{remark}[Endpoint delays]
The restriction $\tau_d\in(0,\tau)$ avoids truncating the mollifier at an endpoint.  The case $\tau_d=\tau$ can be treated using a one-sided approximate identity.  Concentration at $\tau_d=0$ is an instantaneous limit, for which the limiting atom must be split off and absorbed into the instantaneous form as described in \Cref{rem:atom-zero}.  This limit is non-uniform in general.  Indeed, \Cref{thm:small-delay-instability} constructs fixed-mass discrete delays approaching zero, together with a sequence of uniformly bounded data, such that the corresponding solutions lose uniform control at the critical threshold and grow exponentially in the supercritical regime.
\end{remark}


\begin{corollary}[Moving multiple delays]\label[corollary]{cor:moving-delays}
Let
\begin{align*}
  \mu_n=\sum_{j=1}^Jm_j\delta_{\tau_{j,n}},
  \quad
  \mu=\sum_{j=1}^Jm_j\delta_{\tau_j},
\end{align*}
where $m_j>0$,
$\tau_{j,n},\tau_j\in(0,\tau]$, and
$\tau_{j,n}\to\tau_j$.  Let the corresponding solutions have common data.  Then,
\begin{align*}
  u_{\mu_n}\longrightarrow u_\mu
  \quad\text{in }\W(0,\Tend)\cap \mathcal{C}([0,\Tend];H).
\end{align*}
If the joined limit trajectory admits an extension in
$H^\theta(\R;V)$ for some $0<\theta\le1$, then, for all sufficiently large $n$,
\begin{align}\label{eq:moving-delay-rate}
  \norm{u_{\mu_n}-u_\mu}_{\W(0,\Tend)}
  +\norm{u_{\mu_n}-u_\mu}_{\mathcal{C}([0,\Tend];H)}
  \le C
  \left(\sum_{j=1}^Jm_j\abs{\tau_{j,n}-\tau_j}\right)^\theta,
\end{align}
where $C$ is independent of $n$.
\end{corollary}

\begin{proof}
We set $M:=\sum_{j=1}^Jm_j$ and $\delta:=\frac12\min_{1\le j\le J}\tau_j>0$. There is $n_0$ such that $\tau_{j,n}\ge\delta$ for all $j$ and $n\ge n_0$, so that $\kappa_{\mu_n}(\beta)\le Me^{-\beta\delta}$.  Choosing $\beta\ge0$ with $\Lambda_1Me^{-\beta\delta}<\alpha_0/2$ gives
\begin{align}\label{eq:moving-delay-common-margin}
  d_{\mu_n,\beta}\ge\frac{\alpha_0}{2},
  \quad n\ge n_0.
\end{align}
The plan $\pi_n:=\sum_{j=1}^Jm_j\delta_{(\tau_{j,n},\tau_j)}$ belongs to $\Pi(\mu_n,\mu)$.  For the zero extension $G_0$ of $\tilde u_\mu$,
\begin{align*}
  \mathfrak C_{G_0}(\mu_n,\mu)
  \le\sum_{j=1}^Jm_j\,\omega_{G_0}(\abs{\tau_{j,n}-\tau_j})
  \longrightarrow0,
\end{align*}
and \eqref{eq:transport-modulus-full} with \eqref{eq:moving-delay-common-margin} gives the convergence.  If $\tilde u_\mu$ admits an extension $G_\theta\in H^\theta(\R;V)$, the same plan gives $W_1(\mu_n,\mu)\le\sum_{j=1}^Jm_j\abs{\tau_{j,n}-\tau_j}$, and \eqref{eq:wasserstein-rate-full} with \eqref{eq:moving-delay-common-margin} proves \eqref{eq:moving-delay-rate} for $n\ge n_0$.
\end{proof}


\begin{corollary}[Finite-delay quadrature approximation of a distributed kernel]\label[corollary]{cor:atomic-approximation}
Let $\mu\in\M_+^0([0,\tau])$ have mass $M>0$, and let $\mu_N\in\M_+^0([0,\tau])$ be finite atomic measures with the same mass $M$. Assume $W_1(\mu_N,\mu)\to0$, and let all solutions have common data. Then,
\begin{align*}
  u_{\mu_N}\longrightarrow u_\mu
  \quad\text{in }\W(0,\Tend)\cap \mathcal{C}([0,\Tend];H).
\end{align*}
If the joined limit trajectory admits an extension in $H^\theta(\R;V)$ for some $0<\theta\le1$, then
\begin{align*}
  \norm{u_{\mu_N}-u_\mu}_{\W(0,\Tend)}
  +\norm{u_{\mu_N}-u_\mu}_{\mathcal{C}([0,\Tend];H)}
  =O\bigl(W_1(\mu_N,\mu)^\theta\bigr).
\end{align*}
Finite atomic approximations with atoms in $(0,\tau]$ and the stated $W_1$ convergence can always be constructed.
\end{corollary}

\begin{proof}
For equal masses on the compact interval $[0,\tau]$, convergence in $W_1$ implies narrow convergence, so \Cref{thm:narrow-stability} gives the convergence.  Fix $d_*\in(0,\alpha_0)$.  Because $\mu(\{0\})=0$, \Cref{lem:common-weight} provides $\beta\ge0$ and $N_0$ with $d_{\mu_N,\beta}\ge d_*$ for $N\ge N_0$.  Then, \eqref{eq:wasserstein-rate-full} with this common weight gives the rate, with constant $e^{\beta\Tend}C_\theta\Lambda_1(1+\sqrt{\alpha_0}+\alpha_0+\Lambda_0)M^{1-\theta}\norm{G_\theta}_{H^\theta(\R;V)}/d_*$.

For the construction, let $h_N:=\tau/N$, $I_{N,1}:=[0,h_N]$, $I_{N,k}:=((k-1)h_N,kh_N]$ for $2\le k\le N$, and
\begin{align*}
  \mu_N:=\sum_{k=1}^N\mu(I_{N,k})\,\delta_{kh_N}.
\end{align*}
Then, $\mu_N\in\M_+^0([0,\tau])$ has mass $M$.  The plan
\begin{align}\label{eq:atomic-approximation-plan}
  \pi_N(E):=\sum_{k=1}^N\int_{I_{N,k}}\boldsymbol{1}_E(kh_N,r)\,\dd\mu(r),
\end{align}
which moves $\mu|_{I_{N,k}}$ to the atom at $kh_N$, has marginals $\mu_N$ and $\mu$, and $\abs{s-r}\le h_N$ on its support. Therefore, $W_1(\mu_N,\mu)\le Mh_N\to0$.
\end{proof}

\begin{remark}[Kernel quadrature as a numerical error component]
\label[remark]{rem:numerical-splitting}
Let $\mu_N$ be a finite atomic approximation of $\mu$ as in \Cref{cor:atomic-approximation}, and let $U_{h,k,N}$ be a spatially and temporally discrete approximation of $u_{\mu_N}$.  Suppose that the joined limit trajectory $\tilde u_\mu$ admits an extension $G_\theta\in H^\theta(\R;V)$ and that there are a common $\beta\ge0$, a constant $d_*>0$, and an index $N_0$ such that
\begin{align*}
  d_{\mu_N,\beta}\ge d_*,
  \quad N\ge N_0.
\end{align*}
Whenever $U_{h,k,N}\in L^2_\beta(0,\Tend;V)$, the triangle inequality and \eqref{eq:wasserstein-rate} give
\begin{align}\label{eq:numerical-error-splitting}
  \norm{u_\mu-U_{h,k,N}}_{L^2_\beta(0,\Tend;V)}
  &\le
  \norm{u_\mu-u_{\mu_N}}_{L^2_\beta(0,\Tend;V)}
  +\norm{u_{\mu_N}-U_{h,k,N}}_{L^2_\beta(0,\Tend;V)}
  \notag\\
  &\le
  C W_1(\mu_N,\mu)^\theta
  +\norm{u_{\mu_N}-U_{h,k,N}}_{L^2_\beta(0,\Tend;V)},
\end{align}
where $C$ is independent of $h$, $k$, and $N$. If $U_{h,k,N}\in\W(0,\Tend)$, then the triangle inequality and \eqref{eq:wasserstein-rate-full} similarly give
\begin{align*}
  &\norm{u_\mu-U_{h,k,N}}_{\W(0,\Tend)}
  +\norm{u_\mu-U_{h,k,N}}_{\mathcal C([0,\Tend];H)}
  \\
  &\quad\le
  C W_1(\mu_N,\mu)^\theta
  +\norm{u_{\mu_N}-U_{h,k,N}}_{\W(0,\Tend)}
  +\norm{u_{\mu_N}-U_{h,k,N}}_{\mathcal C([0,\Tend];H)}.
\end{align*}
Thus, the first term in \eqref{eq:numerical-error-splitting} is a kernel-quadrature or model-reduction error, whereas the second is the usual finite-element and time-stepping error for the finite-delay problem. The estimate is independent of the particular state discretisation and can therefore be combined with standard parabolic and delay-equation schemes~\cite{BellenZennaro2003,Pao1999,Thomee2006}. A mesh-uniform semidiscrete application requires additional temporal regularity estimates, as explained in \Cref{sec:spatial-transfer}, and a complete fully discrete analysis also requires spatial-consistency and time-stepping bounds that remain uniform over the chosen family of atomic kernels.  For delays acting on the state through a bounded operator, matching lower bounds for midpoint quadrature of the delay measure are proved in~\cite[Theorem~8.5]{Shikhman2026b}.  For the principal-part delays considered here, \Cref{cor:atomic-approximation} provides only the upper rate $W_1(\mu_N,\mu)^\theta$; no matching lower bound for quadrature is claimed.
\end{remark}

\begin{remark}[Delay atoms and Prony modes are different objects]
\label[remark]{rem:not-prony}
If
\begin{align*}
  \mu=\sum_{j=1}^Jm_j\delta_{\tau_j},
\end{align*}
then
\begin{align*}
  (\K_\mu\tilde u)(t)
  =
  \sum_{j=1}^Jm_jA_1\tilde u(t-\tau_j).
\end{align*}
Thus, an atomic measure in the delay variable produces a finite sum of shifted states.  It should not be called a Prony kernel.  By contrast, a finite Prony-series memory kernel has a lag-time density of the form
\begin{align*}
  k(t)=\sum_{\ell=1}^Lc_\ell e^{-\gamma_\ell t},
  \quad t>0,
\end{align*}
and, in the positive-type setting, corresponds to the atomic measure
\begin{align*}
  \rho=\sum_{\ell=1}^Lc_\ell\delta_{\gamma_\ell}
\end{align*}
in the separate relaxation-rate variable. Keeping the delay measure $\mu$ and the relaxation-rate measure $\rho$ distinct is essential when this paper is read together with the companion work on completely monotone memory~\cite{IshizakaCoercivity2026}.
\end{remark}

\begin{remark}[Vanishing memory]
If $\mu_n\in\M_+^0([0,\tau])$ and $\mu_n([0,\tau])\to0$, the solutions converge to the memory-free solution with the same data in $\W(0,\Tend)\cap\mathcal C([0,\Tend];H)$, at a rate proportional to $\mu_n([0,\tau])$.  This follows from the local total-variation Lipschitz estimate of~\cite[Corollary~5.3]{IshizakaMixed2026} at the zero kernel and requires no transport argument.
\end{remark}

\subsection{Numerical interpretation}\label{sec:spatial-transfer}
The preceding estimates provide a separate bound for the error caused by replacing one delay law by another.  For any fixed spatial Galerkin space on which the abstract hypotheses hold, the same residual argument applies to the associated finite-dimensional evolution problem.  This elementary observation does not by itself yield a \emph{mesh-uniform} transport estimate.  Uniformity additionally requires common constants in the discrete form bounds, a common admissible exponential weight, and a bound of the form
\begin{align}\label{eq:uniform-galerkin-conditions}
  \sup_h \norm{G_h}_{H^\theta(\R;V_h)}<\infty
\end{align}
for the joined reference trajectories.  The last condition is a substantive regularity requirement and is not established here from general Galerkin data.

\begin{remark}[Scope of a numerical sequel]\label[remark]{rem:finite-element-bridge}
Conforming or nonconforming finite-element spaces may provide uniform discrete Poincar\'e inequalities, including the semiregular Crouzeix--Raviart setting of \cite{IshizakaDiscreteSobolev2026}.  This controls the spatial embedding constant but does not establish \eqref{eq:uniform-galerkin-conditions}. Verifying temporal regularity uniformly for suitable projected forcing and history data, together with spatial consistency and time stepping, is reserved for a numerical sequel.
\end{remark}

\section{The singular small-delay boundary}\label{sec:small-delay}
The concentration results above concern a fixed positive limiting lag.  At the origin, an atom acts without delay and changes the instantaneous form itself.  The weighted estimate already detects this distinction.  For $\mu=M\delta_{\tau_d}$,
\begin{align}\label{eq:atomic-weighted-mass}
  \kappa_\mu(\beta)=M e^{-\beta\tau_d},
  \quad
  q_{\mu,\beta}
  =\frac{\Lambda_1M}{\alpha_0}e^{-\beta\tau_d}.
\end{align}
Fix $q_0\in(0,1)$.  If $\Lambda_1M>\alpha_0q_0$, then the requirement $q_{\mu,\beta}\le q_0$ is equivalent to
\begin{align}\label{eq:small-delay-required-weight}
  \beta\ge
  \frac{1}{\tau_d}
  \log\left(\frac{\Lambda_1M}{\alpha_0q_0}\right).
\end{align}
Thus, once the unweighted contraction threshold is exceeded, a contraction margin uniform as $\tau_d\downarrow0$ requires a weight of order $\tau_d^{-1}$.

Principal-part delays in parabolic equations have been studied in \cite{YongPan1993,DiBlasioKunischSinestrari1984,KhusainovPokojovyRacke2015}, while delayed heat models can exhibit instability and ill-posedness \cite{JordanDaiMickens2008,DreherQuintanillaRacke2009}. For nonlinear measure-valued delayed diffusion, concentration towards the origin is treated in \cite{Tsukamoto2026} under pathwise coercivity. Small-delay robustness is also classical in infinite-dimensional feedback systems \cite{Datko1988,LogemannRebarber1996,LogemannRebarberWeiss1996,NicaisePignotti2006}.

Here, the instantaneous principal part remains coercive. We identify the sharp class-level threshold for finite-time energy estimates depending only on the structural constants and delayed mass, uniformly as the lag tends to zero. In the diffusion model below, every fixed scalar mode is spectrally stable at the critical ratio $\Lambda_1M/\alpha_0=1$, yet uniform $L^2(0,\Tend;V)$ control fails in a joint small-delay, high-frequency limit. Above this threshold, we obtain the exact asymptotic logarithmic growth rate for an explicit solution family. The characteristic equation is classical \cite{Hayes1950}; the required root expansion follows directly from the complex implicit-function theorem.

\begin{proposition}[Uniform unweighted bounds below the mass threshold]
\label[proposition]{prop:subcritical-mass-bound}
Let $\mu\in\M_+^0([0,\tau])$ have mass $M>0$, and assume
\begin{align}\label{eq:strict-subcritical-threshold}
  \Lambda_1M<\alpha_0.
\end{align}
Then, the solution of \eqref{eq:problem} satisfies
\begin{align}\label{eq:subcritical-mass-bound}
  &\norm{u_\mu}_{\mathcal C([0,\Tend];H)}^2
  +\norm{u_\mu}_{L^2(0,\Tend;V)}^2 \notag\\
  &\quad\le C\left(
    \norm{u_0}_H^2
    +\norm{f}_{L^2(0,\Tend;V^*)}^2
    +\norm{\psi}_{L^2(-\tau,0;V)}^2
  \right),
\end{align}
where $C$ depends only on $\Tend,\alpha_0,\Lambda_0,\Lambda_1$, and $M$, and is independent of the location and distribution of the delay mass.
\end{proposition}

\begin{proof}
We first obtain the $L^2(0,\Tend;V)$ bound.  For $\beta=0$,
\begin{align*}
  q_{\mu,0}=\frac{\Lambda_1M}{\alpha_0}<1.
\end{align*}
Therefore, \eqref{eq:weighted-u-bound} gives a bound for $\norm{u_\mu}_{L^2(0,\Tend;V)}$ with a constant independent of $\mu$. By \eqref{eq:measure-convolution-bound},
\begin{align*}
  \norm{\K_\mu\tilde u_\mu}_{L^2(0,\Tend;V^*)}
  \le
  \Lambda_1M\left(
    \norm{\psi}_{L^2(-\tau,0;V)}
    +\norm{u_\mu}_{L^2(0,\Tend;V)}
  \right).
\end{align*}
The equation then bounds $\norm{\partial_tu_\mu}_{L^2(0,\Tend;V^*)}$ by the same data. Finally, the standard energy identity gives
\begin{align*}
  \sup_{0\le t\le\Tend}\norm{u_\mu(t)}_H^2
  \le
  \norm{u_0}_H^2
  +2\norm{\partial_tu_\mu}_{L^2(0,\Tend;V^*)}
   \norm{u_\mu}_{L^2(0,\Tend;V)}.
\end{align*}
Combining the three estimates proves \eqref{eq:subcritical-mass-bound}.
\end{proof}

\begin{lemma}[Characteristic roots at the mass threshold]
\label[lemma]{lem:threshold-characteristic-roots}
For $a>0$ and $\mathfrak m>0$, we set
\begin{align*}
  \Phi_{a,\mathfrak m}(z)
  :=z+a+a\mathfrak m e^{-z},
  \quad z\in\mathbb C.
\end{align*}
If $0<\mathfrak m\le1$, then $\Phi_{a,\mathfrak m}$ has no zero in $\{z\in\mathbb C:\operatorname{Re}z\ge0\}$. Furthermore, for each $\mathfrak m\ge1$ and all sufficiently large $a$, $\Phi_{a,\mathfrak m}$ has a zero $z_{\mathfrak m}(a)$ satisfying
\begin{align}\label{eq:threshold-root-expansion}
  z_{\mathfrak m}(a)
  &=
  z_{\mathfrak m}^0
  +\frac{z_{\mathfrak m}^1}{a}
  +\frac{z_{\mathfrak m}^2}{a^2}
  +O(a^{-3}),\\
  z_{\mathfrak m}^0
  &:=
  \log\mathfrak m+\mathrm{i}\pi,\quad
  z_{\mathfrak m}^1:=-z_{\mathfrak m}^0,\quad
  z_{\mathfrak m}^2:=
  z_{\mathfrak m}^0+\frac{(z_{\mathfrak m}^0)^2}{2}.
  \notag
\end{align}
In particular,
\begin{align}\label{eq:threshold-root-real-parts}
  \operatorname{Re}z_{\mathfrak m}(a)
  &\longrightarrow\log\mathfrak m
  &&(\mathfrak m>1),\\
  \operatorname{Re}z_1(a)
  &=-\frac{\pi^2}{2a^2}+O(a^{-3}).
  \notag
\end{align}
\end{lemma}

\begin{proof}
We first exclude roots in the closed right half-plane.  Suppose that
$\Phi_{a,\mathfrak m}(z)=0$ and
$\operatorname{Re}z\ge0$.  Then,
\begin{align*}
  \abs{z+a}
  =a\mathfrak m e^{-\operatorname{Re}z}
  \le a\mathfrak m,
\end{align*}
whereas
\begin{align*}
  \abs{z+a}\ge\operatorname{Re}(z+a)\ge a.
\end{align*}
This is impossible if $\mathfrak m<1$.  If $\mathfrak m=1$, equality throughout would force $z=0$, but $\Phi_{a,1}(0)=2a\neq0$.

We next derive the expansion. We set $\varepsilon:=a^{-1}$ and define
\begin{align*}
  F_{\mathfrak m}(z,\varepsilon)
  :=1+\mathfrak m e^{-z}+\varepsilon z.
\end{align*}
At $z=z_{\mathfrak m}^0=\log\mathfrak m+\mathrm{i}\pi$ and $\varepsilon=0$, one has
\begin{align*}
  F_{\mathfrak m}(z_{\mathfrak m}^0,0)=0,
  \quad
  \partial_zF_{\mathfrak m}(z_{\mathfrak m}^0,0)
  =-\mathfrak m e^{-z_{\mathfrak m}^0}=1.
\end{align*}
The complex implicit-function theorem gives an analytic branch of zeros $z_{\mathfrak m}(\varepsilon)$ near $\varepsilon=0$.  Substitution of
\begin{align*}
  z_{\mathfrak m}(\varepsilon)
  =
  z_{\mathfrak m}^0
  +z_{\mathfrak m}^1\varepsilon
  +z_{\mathfrak m}^2\varepsilon^2
  +O(\varepsilon^3)
\end{align*}
into $F_{\mathfrak m}(z_{\mathfrak m}(\varepsilon),\varepsilon)=0$ and comparison of equal powers of $\varepsilon$ give the coefficients in \eqref{eq:threshold-root-expansion}.  Taking real parts proves \eqref{eq:threshold-root-real-parts}.
\end{proof}

\begin{theorem}[Critical degeneration and supercritical instability]
\label[theorem]{thm:small-delay-instability}
Let $\Tend>0$, $\tau>0$, and $\alpha_0,\Lambda_1,M>0$.  In the delayed-diffusion realisation of \Cref{sec:pde-realisation}, we set
\begin{align*}
  \Omega=(0,\pi),\quad
  H=L^2(0,\pi),\quad
  V=H_0^1(0,\pi),\\
  B_0=\alpha_0,\quad
  B_1=\Lambda_1,
\end{align*}
use $\norm{v}_V=\norm{v'}_{L^2(0,\pi)}$, and set $f=0$. We define
\begin{align*}
  \mathfrak m:=\frac{\Lambda_1M}{\alpha_0}.
\end{align*}
If $\mathfrak m\ge1$, then there are $j_0\in\mathbb N$ and discrete retarded kernels
\begin{align}\label{eq:small-delay-kernels}
  \mu_j:=M\delta_{\tau_j},
  \quad
  \tau_j:=\frac{1}{\alpha_0j}\downarrow0,
  \quad j\ge j_0,
\end{align}
and data $u_{0,j}\in H$ and $\psi_j\in L^2(-\tau,0;V)$ such that
\begin{align}\label{eq:small-delay-data-bound}
  \sup_{j\ge j_0}\left(
    \norm{u_{0,j}}_H^2
    +\norm{\psi_j}_{L^2(-\tau,0;V)}^2
  \right)<\infty.
\end{align}
The corresponding weak solutions have the following properties.
\begin{enumerate}
\item If $\mathfrak m=1$, then, for all sufficiently large $j$,
\begin{align}\label{eq:critical-small-delay-growth}
  \norm{u_j}_{L^2(0,\Tend;V)}^2
  \ge\frac{\Tend}{8}j.
\end{align}
\item If $\mathfrak m>1$, then
\begin{align}\label{eq:small-delay-growth}
  \lim_{j\to\infty}
  \frac{\tau_j}{\Tend}
  \log\norm{u_j}_{\mathcal C([0,\Tend];H)}
  =\log\mathfrak m.
\end{align}
In particular, $\norm{u_j}_{\mathcal C([0,\Tend];H)}\to\infty$.
\end{enumerate}
Consequently, when $\Lambda_1M\ge\alpha_0$, there is no constant depending only on $\Tend,\tau,\alpha_0,\Lambda_0,\Lambda_1$, and $M$ for which
\begin{align}\label{eq:impossible-mass-only-bound}
  &\norm{u}_{\mathcal C([0,\Tend];H)}^2
   +\norm{u}_{L^2(0,\Tend;V)}^2 \notag\\
  &\quad\le C\left(
      \norm{u_0}_H^2
      +\norm{f}_{L^2(0,\Tend;V^*)}^2
      +\norm{\psi}_{L^2(-\tau,0;V)}^2
    \right)
\end{align}
holds for all admissible data and any non-negative retarded measure $\mu\in\M_+^0([0,\tau])$ of mass $\mu([0,\tau])=M$.
\end{theorem}

\begin{proof}
\emph{Step 1: scalar reduction and construction.}
Let
\begin{align*}
  \rho_j:=j^2,\quad
  \varphi_j(x):=\sqrt{\frac{2}{\pi}}\sin(jx),\quad
  c_j:=j^{-1/2}.
\end{align*}
Then, $\norm{\varphi_j}_H=1$, $\norm{\varphi_j}_V^2=\rho_j$, and
\begin{align}\label{eq:small-delay-eigenmode}
  a_0(\varphi_j,v)
  &=\alpha_0\rho_j(\varphi_j,v)_H,\\
  a_1(\varphi_j,v)
  &=\Lambda_1\rho_j(\varphi_j,v)_H
  \quad(v\in V).
  \notag
\end{align}
A trajectory $u(t)=y(t)\varphi_j$ therefore solves the equation with kernel $\mu_j$ precisely when
\begin{align}\label{eq:small-delay-scalar}
  \dot y(t)
  +\alpha_0\rho_jy(t)
  +\Lambda_1M\rho_jy(t-\tau_j)=0.
\end{align}

Because
\begin{align*}
  \alpha_0\rho_j\tau_j=j,
\end{align*}
\Cref{lem:threshold-characteristic-roots} gives, for all sufficiently large $j$, a zero
\begin{align*}
  z_j:=z_{\mathfrak m}(j)
\end{align*}
of
\begin{align}\label{eq:small-delay-characteristic}
  z+j+j\mathfrak m e^{-z}=0.
\end{align}
We write
\begin{align*}
  x_j:=\operatorname{Re}z_j,\quad
  b_j:=\operatorname{Im}z_j.
\end{align*}
The real-valued function
\begin{align}\label{eq:small-delay-yj}
  y_j(t)
  &:=
  c_j\operatorname{Re}
  \exp\left(z_j\frac{t}{\tau_j}\right)\\
  &=
  c_je^{x_jt/\tau_j}
  \cos\left(b_j\frac{t}{\tau_j}\right)
  \notag
\end{align}
solves \eqref{eq:small-delay-scalar} on $\mathbb R$.

We set
\begin{align*}
  u_j(t):=y_j(t)\varphi_j,\quad
  u_{0,j}:=c_j\varphi_j,
  \quad t\in[0,\Tend],
\end{align*}
and, on the common past interval $(-\tau,0)$, define
\begin{align}\label{eq:small-delay-history}
  \psi_j(t):=
  \begin{cases}
    0,&-\tau<t\le-\tau_j,\\
    y_j(t)\varphi_j,&-\tau_j<t<0.
  \end{cases}
\end{align}
For all sufficiently large $j$, $\tau_j<\tau$.  If $0<t<\tau_j$, then $t-\tau_j\in(-\tau_j,0)$ and the delayed value is given by \eqref{eq:small-delay-history}; if $\tau_j\le t<\Tend$, then $t-\tau_j\in[0,\Tend)$ and the definition of $u_j$ gives the same expression. Therefore, in either case,
\begin{align*}
  \tilde u_j(t-\tau_j)
  =y_j(t-\tau_j)\varphi_j
  \quad\text{for a.e. }t\in(0,\Tend).
\end{align*}
Together with \eqref{eq:small-delay-eigenmode}--\eqref{eq:small-delay-history}, this shows that $u_j$ is the weak solution associated with $(u_{0,j},0,\psi_j,\mu_j)$.

\emph{Step 2: uniform bound for the data.}
The sequence $(x_j)$ is bounded.  Therefore,
\begin{align*}
  \norm{\psi_j}_{L^2(-\tau,0;V)}^2
  &=
  \rho_jc_j^2
  \int_{-\tau_j}^{0}
  e^{2x_jt/\tau_j}
  \cos^2\left(b_j\frac{t}{\tau_j}\right)\,\dd t\\
  &\le
  \rho_jc_j^2\tau_j
  \max\{1,e^{-2x_j}\}\\
  &=
  \frac{1}{\alpha_0}\max\{1,e^{-2x_j}\}.
\end{align*}
Together with $\norm{u_{0,j}}_H^2=c_j^2=j^{-1}$, this proves \eqref{eq:small-delay-data-bound}.

\emph{Step 3: the critical case.}
Suppose that $\mathfrak m=1$.  By \eqref{eq:threshold-root-real-parts},
\begin{align*}
  \frac{x_j}{\tau_j}\longrightarrow0,
  \quad
  \frac{b_j}{\tau_j}\longrightarrow\infty.
\end{align*}
Therefore, for all sufficiently large $j$,
\begin{align*}
  e^{2x_jt/\tau_j}\ge\frac12
  \quad(0\le t\le\Tend),
\end{align*}
and
\begin{align*}
  \int_0^\Tend
  \cos^2\left(b_j\frac{t}{\tau_j}\right)\,\dd t
  \ge\frac{\Tend}{4}.
\end{align*}
Consequently,
\begin{align*}
  \norm{u_j}_{L^2(0,\Tend;V)}^2
  &=
  \rho_jc_j^2
  \int_0^\Tend
  e^{2x_jt/\tau_j}
  \cos^2\left(b_j\frac{t}{\tau_j}\right)\,\dd t\\
  &\ge\frac{\Tend}{8}j,
\end{align*}
which proves \eqref{eq:critical-small-delay-growth}.

\emph{Step 4: the supercritical case.}
Suppose that $\mathfrak m>1$.  Then, $x_j\to\log\mathfrak m>0$ and $b_j\to\pi$.  The upper bound
\begin{align*}
  \norm{u_j}_{\mathcal C([0,\Tend];H)}
  \le c_j\exp\left(x_j\frac{\Tend}{\tau_j}\right)
\end{align*}
is immediate.  Let
\begin{align*}
  k_j
  :=
  \left\lfloor\frac{b_j\Tend}{\pi\tau_j}\right\rfloor,
  \quad
  t_j:=\frac{k_j\pi\tau_j}{b_j}.
\end{align*}
For all sufficiently large $j$,
\begin{align*}
  \Tend-\frac{\pi\tau_j}{b_j}<t_j\le\Tend
\end{align*}
and the cosine in \eqref{eq:small-delay-yj} has absolute value one at
$t_j$.  Thus,
\begin{align*}
  c_j\exp\left(
    x_j\frac{\Tend}{\tau_j}-\frac{\pi x_j}{b_j}
  \right)
  \le
  \norm{u_j}_{\mathcal C([0,\Tend];H)}
  \le
  c_j\exp\left(x_j\frac{\Tend}{\tau_j}\right).
\end{align*}
Taking logarithms, multiplying by $\tau_j/\Tend$, and using
\begin{align*}
  \tau_j\abs{\log c_j}\longrightarrow0
\end{align*}
proves \eqref{eq:small-delay-growth}.  The failure of \eqref{eq:impossible-mass-only-bound} follows in the critical case from \eqref{eq:critical-small-delay-growth} and in the supercritical case from \eqref{eq:small-delay-growth}.
\end{proof}

\begin{remark}[The critical value is spectrally stable but not uniform]
\label[remark]{rem:critical-spectral-stability}
For $\mathfrak m=1$ and any fixed $a>0$, \Cref{lem:threshold-characteristic-roots} places all characteristic roots strictly in the open left half-plane.  The critical failure in \Cref{thm:small-delay-instability} is therefore not caused by an unstable root at any fixed mode.  Instead, a pair of roots approaches $\{\pm\mathrm{i}\pi\}$ in the dimensionless variable while the spatial frequency diverges.  This destroys uniform control of the $L^2(0,\Tend;V)$ norm even though each fixed scalar delay equation is spectrally stable.
\end{remark}

\begin{remark}[Why the weighted estimate has the correct supercritical scale]
\label[remark]{rem:weighted-small-delay-scale}
For the family \eqref{eq:small-delay-kernels},
\begin{align*}
  q_{\mu_j,\beta}
  =
  \mathfrak m e^{-\beta\tau_j}.
\end{align*}
Fixing $q_0\in(0,1)$ and requiring
$q_{\mu_j,\beta}\le q_0$ forces
\begin{align*}
  \beta\ge
  \frac{1}{\tau_j}
  \log\left(\frac{\mathfrak m}{q_0}\right).
\end{align*}
When $\mathfrak m>1$, the logarithmic growth coefficient in \eqref{eq:small-delay-growth} is $\log\mathfrak m$, which is the limit of the weighted coefficient as $q_0\uparrow1$.  At the critical value, the loss occurs instead in the coercive space--time norm as described in \Cref{rem:critical-spectral-stability}.
\end{remark}

\begin{remark}[Why positive-lag stability is not contradicted]
\label[remark]{rem:small-delay-no-contradiction}
The measures in \eqref{eq:small-delay-kernels} converge narrowly to $M\delta_0$.  This limit does not belong to $\M_+^0([0,\tau])$: by \Cref{rem:atom-zero}, it changes the instantaneous form from $a_0$ to $a_0+Ma_1$.  Furthermore, the construction follows spatial frequencies $\rho_j\to\infty$ and uses $j$-dependent data.  It therefore does not contradict \Cref{thm:narrow-stability} and \Cref{cor:delay-concentration}: the former is a fixed-data result for a retarded limiting kernel and uses a common admissible weighted margin, while the latter estimates translations of a fixed reference trajectory. In the present construction, the limiting measure is not retarded, the data vary with $j$, and no common admissible weighted margin exists.
\end{remark}

\section{Solution-level sharpness of the fractional exponent}
\label{sec:sharpness}
We show that the exponent in \Cref{thm:wasserstein-rate} is optimal not only for the translation estimate of \Cref{lem:translation}, as the family constructed below also shows, but for the \emph{solution difference} itself, on the class of admissible reference trajectories with uniformly bounded $H^\theta$ norm.  Using a spatial frequency larger than the temporal frequency removes the putative gain from parabolic smoothing. \Cref{cor:two-sided-modulus} combines the resulting lower bound with \Cref{thm:wasserstein-rate} into a two-sided estimate.  For $\theta=0$, we use the conventions $H^0(\R;V)=L^2(\R;V)$ and $W_1(\mu,\nu)^0:=1$.

\begin{theorem}[Solution-level optimality for moving delay atoms]
\label[theorem]{thm:solution-sharpness}
Fix $\Tend>0$, $0<\tau_d<\min\{\Tend,\tau\}$ and $\alpha_0,\Lambda_1>0$.  Choose a constant $0<m_0<\alpha_0/\Lambda_1$ and let $0\le\theta\le1$. In the delayed-diffusion realisation of \Cref{sec:pde-realisation}, we set
\begin{align*}
  \Omega=(0,\pi),\quad
  V=H_0^1(0,\pi),\quad H=L^2(0,\pi),\quad
  B_0=\alpha_0,\quad B_1=\Lambda_1.
\end{align*}
There are positive constants $c$, $C$, and $h_0$ with the following property.  For any $m\in(0,m_0]$ and any $h\in(0,h_0]$, the retarded kernels
\begin{align*}
  \nu=m\delta_{\tau_d},\quad
  \mu_h=m\delta_{\tau_d+h},
  \quad W_1(\mu_h,\nu)=mh,
\end{align*}
belong to $\M_+^0([0,\tau])$, and there are common data $u_0=0$, $\psi=0$, and $f_h\in L^2(0,\Tend;V^*)$ such that
\begin{align}\label{eq:solution-sharpness-uniform}
  \norm{f_h}_{L^2(0,\Tend;V^*)}\le C,
  \quad
  \norm{G_{\nu,h}}_{H^\theta(\R;V)}\le C,
\end{align}
where $G_{\nu,h}$ is the zero extension of the reference solution $u_{\nu,h}$, and
\begin{align}\label{eq:solution-level-lower}
  \norm{u_{\mu_h,h}-u_{\nu,h}}_{L^2(0,\Tend;V)}
  \ge c\,m^{1-\theta}W_1(\mu_h,\nu)^\theta.
\end{align}
For $\theta=0$, the right-hand side is $cm$. Consequently, even when the reference trajectories and forcings obey \eqref{eq:solution-sharpness-uniform}, the exponent $\theta$ in the $W_1^\theta$ solution estimate cannot be replaced uniformly by any larger exponent.  The construction uses spatial modes and data that depend on $h$; no obstruction to a better rate for a fixed smoother solution is asserted.
\end{theorem}

\begin{proof}
We choose a closed interval $J=[a,b]$ with
\begin{align*}
  0<a<b<\min\{\tau_d,\Tend-\tau_d\},
\end{align*}
a number $\delta\in(0,a)$, and a smooth cutoff $\chi\in C_c^\infty((0,\min\{\tau_d,\Tend-\tau_d\}))$ that equals one on $[a-\delta,b]$.  For $h>0$, put $N:=\pi/h$, choose an integer $j\ge N$, and set $\rho:=j^2$ and
\begin{align}\label{eq:sharp-reference-data}
  \phi(x)&:=\frac1{j}\sqrt{\frac2\pi}\sin(jx),\\
  g(t)&:=N^{-\theta}\chi(t)\sin(Nt),\quad
  G(t,x):=g(t)\phi(x).\notag
\end{align}
These objects depend on $h$; we suppress this dependence in the notation. The Dirichlet eigenmode $\phi$ satisfies
\begin{align}\label{eq:sharp-eigen-scales}
  \norm{\phi}_V=1,\quad
  \norm{\phi}_{V^*}=\rho^{-1},\quad
  A\phi=\rho\phi\quad\text{in }V^*,
\end{align}
where $A=-\partial_{xx}$ is the Dirichlet Laplacian.  For $N\ge1$, the usual Fourier-shift estimate for the smooth compactly supported function $\chi$ and a direct computation give, with a constant $C_\chi$ depending only on $\chi$ and $\theta$,
\begin{align}\label{eq:sharp-Htheta}
  \norm{G}_{H^\theta(\R;V)}
  =\norm{g}_{H^\theta(\R)}\le C_\chi,
  \quad
  \norm{g'}_{L^2(\R)}\le C_\chi N^{1-\theta}.
\end{align}
We set
\begin{align*}
  c_0:=\frac{\Lambda_1}{\alpha_0}(b-a)^{1/2},
  \quad
  C_1:=\frac{2\Lambda_1C_\chi}{\alpha_0^2},
\end{align*}
and define
\begin{align*}
  N_0:=\max\Bigl\{1,\frac{2}{b-a},\frac{2C_1}{c_0}\Bigr\},
  \quad
  h_0:=\min\Bigl\{\delta,\tau-\tau_d,\frac{\pi}{N_0}\Bigr\}.
\end{align*}
These numbers depend neither on $m$ nor on $h$.  Let $m\in(0,m_0]$ and $h\in(0,h_0]$.  Then, $N\ge N_0$ and $\tau_d+h\le\tau$, and hence $\mu_h\in\M_+^0([0,\tau])$.

We define the forcing on $(0,\Tend)$ by
\begin{align}\label{eq:sharp-forcing}
  f_h(t):=\partial_tG(t)+\alpha_0AG(t)
  +\Lambda_1mA G(t-\tau_d).
\end{align}
Take the initial condition and past history to be zero.  Then, the reference solution is exactly $u_{\nu,h}=G|_{(0,\Tend)}$, and $G_{\nu,h}=G$.  By \eqref{eq:sharp-eigen-scales} and \eqref{eq:sharp-Htheta},
\begin{align*}
 \norm{f_h}_{L^2(0,\Tend;V^*)}
 &\le \rho^{-1}\norm{g'}_{L^2(\R)}
   +(\alpha_0+\Lambda_1m_0)\norm{g}_{L^2(\R)}\\
 &\le C_\chi N^{1-\theta}\rho^{-1}
   +(\alpha_0+\Lambda_1m_0)\norm{\chi}_{L^2(\R)}N^{-\theta}\\
 &\le C_\chi+(\alpha_0+\Lambda_1m_0)\norm{\chi}_{L^2(\R)},
\end{align*}
because $\rho\ge N^2$ and $N\ge1$.  Thus, both bounds in \eqref{eq:solution-sharpness-uniform} hold with a constant $C$ independent of $m\in(0,m_0]$ and $h\in(0,h_0]$.

Let $e=u_{\mu_h,h}-u_{\nu,h}$.  Uniqueness and the common eigenmode show that $e(t)=y(t)\phi$ for some scalar $y$ satisfying
\begin{align}\label{eq:sharp-error-scalar}
  y'(t)+\alpha_0\rho y(t)
  +\Lambda_1m\rho\,\widetilde y(t-\tau_d-h)
  =-\Lambda_1m\rho D(t),
\end{align}
where
\begin{align*}
 D(t):=g(t-\tau_d-h)-g(t-\tau_d),
 \quad y(0)=0,
\end{align*}
and the past of $y$ vanishes.  The support of $g$ begins strictly after zero, so $D$ and hence $e$ vanish on $(0,\tau_d]$.  On the interval $0<t<\tau_d+b$, the delayed error in \eqref{eq:sharp-error-scalar} is therefore zero: indeed, $t-\tau_d-h<b<\tau_d$.  Consequently, on this interval $y$ obeys the ordinary scalar parabolic equation
\begin{align}\label{eq:sharp-first-window-ode}
  y'+\alpha_0\rho y
  =-\Lambda_1m\rho D,
  \quad y(0)=0.
\end{align}
We set
\begin{align*}
  z(t):=-\frac{\Lambda_1m}{\alpha_0}D(t),
  \quad w:=y-z.
\end{align*}
Because $z(0)=0$, subtraction from \eqref{eq:sharp-first-window-ode} yields
\begin{align*}
  w'+\alpha_0\rho w=-z',\quad w(0)=0.
\end{align*}
Young's convolution inequality with the scalar heat kernel and \eqref{eq:sharp-Htheta} give
\begin{align}\label{eq:sharp-ode-remainder}
 \norm{w}_{L^2(0,\tau_d+b)}
 &\le \frac{1}{\alpha_0\rho}
    \norm{z'}_{L^2(0,\tau_d+b)}
  \le \frac{2\Lambda_1m}{\alpha_0^2\rho}\norm{g'}_{L^2(\R)}\\
 &\le C_1m\frac{N^{1-\theta}}{\rho}
  \le C_1mN^{-1-\theta}.\notag
\end{align}
For $t\in\tau_d+J$, both $t-\tau_d$ and $t-\tau_d-h$ lie in $[a-\delta,b]$, because $h\le\delta$. Therefore, both cutoff values equal one and, since $Nh=\pi$,
\begin{align*}
  D(t)=-2N^{-\theta}\sin(N(t-\tau_d)),
   \quad t\in\tau_d+J.
\end{align*}
Because $N\ge2/(b-a)$,
\begin{align*}
 \int_a^b\sin^2(Ns)\,\dd s
 =\frac{b-a}{2}-\frac{\sin(2Nb)-\sin(2Na)}{4N}
 \ge\frac{b-a}{4},
\end{align*}
and therefore
\begin{align}\label{eq:sharp-D-lower}
  \norm{z}_{L^2(\tau_d+J)}
  =\frac{2\Lambda_1m}{\alpha_0}N^{-\theta}
   \left(\int_a^b\sin^2(Ns)\,\dd s\right)^{1/2}
  \ge c_0mN^{-\theta}.
\end{align}
Because $\tau_d+J\subset(0,\tau_d+b]$ and $N\ge2C_1/c_0$, combining \eqref{eq:sharp-ode-remainder} and \eqref{eq:sharp-D-lower} gives
\begin{align*}
  \norm{e}_{L^2(0,\Tend;V)}
  &=\norm{y}_{L^2(0,\Tend)}
  \ge\norm{y}_{L^2(\tau_d+J)}\\
  &\ge c_0mN^{-\theta}-C_1mN^{-1-\theta}
  \ge \frac{c_0}{2}mN^{-\theta}.
\end{align*}
Finally, $W_1(\mu_h,\nu)=mh=m\pi/N$, and hence
\begin{align*}
  mN^{-\theta}
  =\pi^{-\theta}m^{1-\theta}W_1(\mu_h,\nu)^\theta.
\end{align*}
This proves \eqref{eq:solution-level-lower} with $c:=\pi^{-\theta}c_0/2$, which is independent of $m$ and $h$.  If a larger exponent held uniformly under the stated bounds, the last inequality and $W_1(\mu_h,\nu)\to0$ as $h\downarrow0$ would contradict it.
\end{proof}

We combine \Cref{thm:solution-sharpness} with the upper bounds of \Cref{sec:rates}.  For $0\le\theta\le1$, $R>0$, and non-negative retarded kernels $\mu,\nu\in\M_+^0([0,\tau])$ of the same mass, set
\begin{align}\label{eq:modulus-def}
  \Delta_\theta(\mu,\nu;R)
  :=\sup\norm{u_\mu-u_\nu}_{L^2(0,\Tend;V)},
\end{align}
where the supremum is taken over all common data $(u_0,f,\psi)$ for which the joined reference trajectory $\tilde u_\nu$ admits an extension $G\in H^\theta(\R;V)$ with $\norm{G}_{H^\theta(\R;V)}\le R$.  The normalisation is through the reference trajectory, not through the data: a bound on $f$ in $L^2(0,\Tend;V^*)$ alone does not control $\norm{G}_{H^\theta(\R;V)}$.

\begin{corollary}[Two-sided modulus for moving delay atoms]
\label[corollary]{cor:two-sided-modulus}
Let $0\le\theta\le1$ and $R>0$, and let $C_\theta$ be the constant of
\Cref{lem:translation}, with $C_0:=2$.
\begin{enumerate}
\item Let $\mu,\nu\in\M_+^0([0,\tau])$ have the same mass $M>0$, and suppose that $\Lambda_1M\le\alpha_0-d_*$ for some $d_*\in(0,\alpha_0)$. Then,
\begin{align}\label{eq:modulus-upper}
  \Delta_\theta(\mu,\nu;R)
  \le\frac{C_\theta\Lambda_1}{d_*}\,
  M^{1-\theta}W_1(\mu,\nu)^\theta R.
\end{align}
\item In the setting of \Cref{thm:solution-sharpness}, there is a constant $c'>0$, independent of $m$, $h$, and $R$, such that, for all $m\in(0,m_0]$, $h\in(0,h_0]$, and $R>0$,
\begin{align}\label{eq:modulus-two-sided}
  c'R\,m^{1-\theta}W_1(\mu_h,\nu)^\theta
  \le\Delta_\theta(\mu_h,\nu;R)
  \le\frac{C_\theta\Lambda_1}{\alpha_0-\Lambda_1m_0}\,
  R\,m^{1-\theta}W_1(\mu_h,\nu)^\theta.
\end{align}
\end{enumerate}
In particular, $\Delta_\theta(\mu_h,\nu;R)$ is comparable with $R\,m\,h^\theta$, uniformly in $m\in(0,m_0]$, $h\in(0,h_0]$, and $R>0$. Part~(1) holds for all equal-mass pairs in the abstract setting, whereas the lower bound in part~(2) concerns the family of moving single atoms in the model problem.
\end{corollary}

\begin{proof}
For part~(1), let $(u_0,f,\psi)$ be admissible common data, and let $G\in H^\theta(\R;V)$ be an extension of $\tilde u_\nu$ with $\norm{G}_{H^\theta(\R;V)}\le R$.  Because $\mu$ is non-negative with mass $M$, $\kappa_\mu(0)=M$, and therefore
\begin{align*}
  d_{\mu,0}=\alpha_0-\Lambda_1M\ge d_*>0.
\end{align*}
If $0<\theta\le1$, \eqref{eq:wasserstein-rate} with $\beta=0$ gives
\begin{align*}
  \norm{u_\mu-u_\nu}_{L^2(0,\Tend;V)}
  \le\frac{C_\theta\Lambda_1}{d_*}M^{1-\theta}W_1(\mu,\nu)^\theta R.
\end{align*}
If $\theta=0$, then \eqref{eq:translation-modulus-bound} and \Cref{lem:transport-cost} give $\mathfrak C_G(\mu,\nu)\le2M\norm{G}_{L^2(\R;V)}\le2MR$, and \eqref{eq:transport-modulus-stability} with $\beta=0$ gives the same bound with $C_0=2$.  Taking the supremum over admissible data proves \eqref{eq:modulus-upper}.

For part~(2), the upper bound follows from part~(1) with $M=m$ and $d_*:=\alpha_0-\Lambda_1m_0$, because $\Lambda_1m\le\Lambda_1m_0$.  For the lower bound, let $C$ and $c$ be the constants of \Cref{thm:solution-sharpness}, and replace the forcing $f_h$ by $(R/C)f_h$.  By linearity and uniqueness, the reference solution, its zero extension, and the solution difference are multiplied by $R/C$.  The rescaled zero extension therefore has $H^\theta(\R;V)$ norm at most $R$, and \eqref{eq:solution-level-lower} gives
\begin{align*}
  \Delta_\theta(\mu_h,\nu;R)
  \ge\frac{R}{C}\,c\,m^{1-\theta}W_1(\mu_h,\nu)^\theta.
\end{align*}
This proves \eqref{eq:modulus-two-sided} with $c':=c/C$.  The restriction $m_0<\alpha_0/\Lambda_1$ is used only for the upper bound; the proof of \Cref{thm:solution-sharpness} does not use it.
\end{proof}

\begin{remark}[Contrast with delays acting on the state]
\label[remark]{rem:contrast-state-delay}
For $\theta=0$, \eqref{eq:modulus-two-sided} gives $\Delta_0(\mu_h,\nu;R)\ge c'Rm$ for any $h\in(0,h_0]$.  Thus, over reference trajectories bounded in $L^2(\R;V)$, the worst-case solution difference does not tend to zero as the displacement $h$ vanishes.  For delays acting on the state through a bounded operator, the situation is different.  In the setting of~\cite{Shikhman2026b}, where the error is measured in $H$ at a fixed time over bounded sets of continuous $H$-valued histories, the worst case for the Dirichlet Laplacian is comparable with $d\sqrt{\log(e/d)}$, $d$ being the bounded-Lipschitz distance, and logarithmic spatial regularity of the histories restores a Lipschitz modulus.  On $[0,\tau]$, this distance and $W_1$ are equivalent for measures of equal mass, with constants depending only on $\tau$.  The two settings differ in the data class and in the norm of the error, so the comparison is structural rather than quantitative.  When the delayed operator has the same spatial order as the instantaneous one, high spatial frequencies of the residual $\K_{\mu-\nu}\tilde u_\nu$ are not damped relative to the solution; the construction in \Cref{thm:solution-sharpness} exploits exactly this, and the rate is then determined by the temporal regularity of the reference trajectory.
\end{remark}

\begin{remark}[Why the first-delay window matters]
The proof of \Cref{thm:solution-sharpness} uses the interval before the delayed \emph{error} can return to the equation.  In that interval the error satisfies \eqref{eq:sharp-first-window-ode}, and replacing its solution by the quasi-static value $z$ incurs only $O(N^{1-\theta}/\rho)$.  The forcing is uniformly bounded in $L^2(0,\Tend;V^*)$, not necessarily in $L^2(0,\Tend;H)$.
\end{remark}

\section{Concluding remarks}
We have developed a quantitative stability theory for measure-valued delay laws.  The residual principle converts a kernel perturbation into time translations of one fixed joined trajectory.  Transport plans then give a trajectory-dependent modulus estimate, and an $H^\theta(\R;V)$ extension yields a $W_1^\theta$ bound.  The gluing criterion identifies the exact interface condition for fractional regularity.  An explicit differentiated-data argument supplies an $H^1$ extension and hence the Lipschitz rate.

At a fixed positive lag, the estimates quantify concentration of a distributed kernel to a discrete delay, motion of finitely many atoms, and finite atomic quadrature of a distributed law.  The origin is qualitatively different.  The strict inequality $\Lambda_1M<\alpha_0$ gives mass-only unweighted bounds uniformly over all delay locations.  At equality, any fixed scalar mode is spectrally stable, but a simultaneous small-delay and high-frequency limit destroys uniform control of the coercive space--time norm.  Above the threshold, uniformly bounded data generate exponentially growing solutions, with dimensionless logarithmic growth tending to $\log(\Lambda_1M/\alpha_0)$.  Thus, the strict subcritical condition is sharp at the class level.  Independently, for moving delay atoms the supremum of the solution difference over reference trajectories of bounded $H^\theta$ norm is comparable with $m^{1-\theta}W_1^\theta$, so that the $W_1^\theta$ rate is sharp at the solution level in both its exponent and its mass scaling. In contrast with delays acting on the state~\cite{Shikhman2026b}, the worst-case difference over reference trajectories bounded in $L^2(\R;V)$ does not tend to zero with the displacement, and the rate is governed by temporal regularity.

The results are finite-time and rely on coercivity of the instantaneous form.  Extensions to semilinear problems and origin limits after renormalising the instantaneous operator remain natural analytical directions.

The concentration estimate also has a numerical interpretation, as noted in \Cref{rem:numerical-splitting}: a kernel-quadrature contribution can be separated from the spatial and temporal discretisation errors. This separation does not automatically give mesh-uniform estimates; see \Cref{sec:spatial-transfer}.  Deriving verified uniform temporal regularity from semidiscrete data, and combining kernel concentration with spatial consistency and time stepping, are tasks for a numerical sequel.  The singular origin limit shows that constants based only on the kernel mass cannot remain uniform in general at the critical and supercritical thresholds.

\section*{Statements and declarations}

\subsection*{Funding}
The author declares that no funds, grants, or other support were received during the preparation of this manuscript.

\subsection*{Data availability}
No datasets were generated or analysed during the current study.

\subsection*{Competing interests}
The author declares no competing interests.

\subsection*{Use of generative AI}
The research questions, mathematical results, and conclusions are the author's own. The author used ChatGPT (OpenAI) for language editing and assistance with checking selected mathematical derivations and manuscript presentation. The author takes full responsibility for the manuscript.

\end{document}